\documentclass[a4paper,11pt]{article}

\usepackage{lmodern}
\usepackage[T1]{fontenc}
\usepackage{hyperref}
\usepackage{amsfonts,amsmath,amssymb,amsopn,amsthm}
\usepackage[all]{xy}
\usepackage{vmargin}
\usepackage{makeidx}
\usepackage{graphicx}
\usepackage{tikz}
\usepackage{comment}
\usepackage{mathabx}
\usepackage{mathbbol}
\title{Twisted differential operators of negative level \\ and prismatic crystals}
\author{Michel Gros, Bernard Le Stum \& Adolfo Quir\'os
\thanks{Supported by grant PGC2018-0953092-B-I00 (MCTU/AEI/FEDER, UE).}
}
\date{Version of \today}
\usepackage[ style=alphabetic,citestyle=alphabetic,sorting=nyt,sortcites=true,autopunct=true,babel=hyphen,hyperref=true,abbreviate=false,backref=true]{biblatex}
\AtEveryBibitem{ \clearfield{url} \clearfield{urldate} \clearfield{review} \clearfield{series} \clearfield{doi} \clearfield{isbn} \clearfield{issn} } 
\defbibheading{bibempty}{}
\DeclareNolabel{\nolabel{\regexp{[\p{Z}]+}}}
\newtheorem{thm}{Theorem}[section]
\newtheorem{prop}[thm] {Proposition}
\newtheorem{cor}[thm] {Corollary}
\newtheorem{lem}[thm] {Lemma}
\newtheorem{dfn}[thm] {Definition}

\newenvironment{xmp}[1][Example]{\begin{trivlist} \item[\hskip \labelsep {\bfseries #1}]}{\end{trivlist}}
\newenvironment{xmps}[1][Examples]{\begin{trivlist} \item[\hskip \labelsep {\bfseries #1}]}{\end{trivlist}}

\newenvironment{rmk}[1][Remark]{\begin{trivlist} \item[\hskip \labelsep {\bfseries #1}]}{\end{trivlist}}
\newenvironment{rmks}[1][Remarks]{\begin{trivlist} \item[\hskip \labelsep {\bfseries #1}]}{\end{trivlist}}
\begin{document}

\maketitle

\begin{abstract}
We introduce twisted differential calculus of negative level and prove a descent theorem:
Frobenius pullback provides an equivalence between finitely presented modules endowed with a topologically quasi-nilpotent twisted connection of level minus one and those of level zero.
We explain how this is related to the existence of a Cartier operator on prismatic crystals. For the sake of readability, we limit ourselves to the case of dimension one.
\end{abstract}

\tableofcontents

\section*{Introduction}
\addcontentsline{toc}{section}{Introduction}

Recently Barghav Bhatt and Peter Scholze have introduced in \cite{BhattScholze19} the beautiful notion of a prism.
It is built on the concept of a thickening which goes back to Alexander Grothendieck in its modern form.
A thickening of a ring $\overline B$ (always commutative) is simply another ring $B$ together with a surjective map $B \to \overline B$.
Alternatively, we may consider the couple $(B, J)$ where $J$ denotes the kernel of the surjection so that $\overline B = B/J$.
The idea behind any theory using this concept is to allow only some particular thickenings with some extra structure in order to get a better grasp on $\overline B$, which is actually the main object of study even if it sometimes disappears from the picture.
Grothendieck used infinitesimal thickenings ($J$ nilpotent) in order to develop calculus over $\overline B$ and Berthelot extended the theory to positive characteristic $p$ by using PD-thickenings (divided powers on $J$).

A \emph{prism} is a particular kind of a thickening.
For example, if we start from $\overline B = \mathbb Z_{p}[\zeta]$ where $\zeta$ is a primitive $p$th root of unity, we may consider the thickening $B = \mathbb Z_{p}[[q-1]]$ obtained by sending $q$ to $\zeta$.
The ideal $J$ is then generated by the $q$-analog $(p)_{q} = 1 + q + \cdots + q^{p-1}$ of $p$.
We will only consider here prisms over this base that we may call $(p)_{q}$-prisms and we will stick to \emph{bounded} prisms (a very light finiteness condition).
More precisely, a prism (in this context) is a $(p)_{q}$-torsion free $(p,q-1)$-adically complete $\mathbb Z_{p}[[q-1]]$-algebra $B$ endowed with a $\delta$-structure\footnote{This is simply a stable version of a frobenius lift.} satisfying $\delta(q) = 0$ (identifying $q$ with its image in $B$) and such that $\overline B = B/(p)_{q}B$ has bounded $p^\infty$-torsion.
An important example from $p$-adic Hodge theory is Fontaine's ring $A_{\mathrm{inf}} := W(\mathcal O_{\mathbb C_{p}}^\flat)$ with $q = [\underline \zeta]$ (Teichm\"uller lifting of a sequence of $p$th roots of $\zeta$) in which case $\overline {A_{\mathrm{inf}}} = \mathcal O_{\mathbb C_{p}}$.
This applies more generally to \emph{perfectoid} integral rings which then correspond to what are called \emph{perfect} prisms.

Putting all infinitesimal thickenings together provides the infinitesimal site on which crystals are defined.
Any crystal\footnote{We only consider crystals with quasi-coherent components.} (say on $X$ smooth over $S$) gives rise to a module endowed with an integrable connection (on $X/S$) and this defines an equivalence of categories.
Actually, in positive characteristic $p$, one uses PD-thickenings and obtain the crystalline site.
A similar pattern occurs with prisms.
In this article, we consider the following situation: we let $R$ be a $(p)_{q}$-prism and $A$ a $(p, q-1)$-adically complete $R$-algebra which admits a coordinate $x$.
There exists a unique $\delta$-structure on $A$ such that $\delta(x)=0$ and a unique endomorphism $\sigma$ such that $\sigma(x) = qx$.
Then, one may consider the prismatic site of $\overline A/R$ and we will show in section \ref{despr} that a prismatic crystal gives rise to an $A$-module endowed with what we call a \emph{twisted connection of level $-1$}.
In general, a twisted connection of level $-m$ on an $A$-module $M$ is an $R$-linear map satisfying the following Leibniz rule:
\[
\forall f \in A, \forall s \in M, \quad \nabla(fs) = (p^m)_{q} s \otimes \mathrm d_{q^{p^m}}f + \sigma^{p^m}(f)\nabla(s).
\]
We will denote the corresponding category by $\mathrm{MIC}_{q}^{(-m)}(A/R)$.
There exists also the notion of a $q$-PD-thickening giving rise to the $q$-crystalline site and we showed in \cite{GrosLeStumQuiros20} that, if $\mathfrak a$ is a $q$-PD-ideal in $A$, then a $q$-crystal on $\overline A/\mathfrak a$ gives rise to an $A$-module endowed with a twisted connection of level $0$.

The main result of the present article is the following: we prove (theorem \ref{mainthm} and corollary \ref{eqmod}) that, if $A'$ denotes the frobenius pullback of $A$, then frobenius descent provides an equivalence of categories between $A'$-modules endowed with a twisted connection of level $-1$ and $A$-modules endowed with a twisted connection of level $0$ (when we focus on finitely presented topologically quasi-nilpotent objects).
Moreover, as we expected in \cite{Gros20}, section 6,  this is related to prisms via a very general Cartier morphism $C$, whose definition is inspired by the  proof of theorem 16.17 of \cite{BhattScholze19},  from the $q$-crystalline topos of a smooth formal scheme $\mathcal X$ to the prismatic topos of its frobenius pull back $\mathcal X'$.
When $\mathcal X = \mathrm{Spf}(A/\mathfrak a)$, in which case $\mathcal X' = \mathrm{Spf}(A'/(p)_{q}A)$, there exists a commutative diagram
\[
\xymatrix{
\{(p)_{q}\mathrm{-prismatic}\ \mathrm{crystals}\ \mathrm{on}\ \mathcal X'/R\} \ar[r]^-{C^{-1}} \ar[d] & \{q\mathrm{-crystals}\ \mathrm{on}\ \mathcal X/R\} \ar[d] \\
\mathrm{MIC}_{q}^{(-1)}(A'/R) \ar[r]^{F^*} & \mathrm{MIC}_{q}^{(0)}(A/R).
}
\]
As explained above, the bottom map is an equivalence (on finitely presented topologically quasi-nilpotent objects) and, introducing flat topology considerations, we will prove in a forthcoming article that  this is also the case for the vertical maps.
Of course, this implies that the upper map is also an equivalence and we expect this to be true more generally for any smooth formal scheme $\mathcal X$.
This would be a twisted arithmetic version of the results obtained by Oyama in \cite{Oyama17} and Xu in \cite{Xu19}.

Let us briefly describe the content of this article.
In the first section, we introduce the ring of twisted divided polynomials of negative level and show in proposition \ref{del1} that it automatically inherits a $\delta$-structure in the case of level $-1$.
In the second section, we explain the basics of twisted calculus on an adic ring and give a fundamental example.
In section three, we develop  twisted calculus of negative level.
In section four, we introduce the level raising functor and show in theorem \ref{mainthm} that this is an equivalence when we move from level $-1$ to level $0$ (Frobenius descent).
In section five, we recall some basic notions on prisms and show in corollary \ref{prismenv} that the ring of twisted divided polynomials of level minus one is the prismatic envelope of the polynomial ring for the \emph{symmetric} $\delta$-structure.
In the last section, we explain Cartier descent from prismatic crystals to $q$-PD-crystals and we show in proposition \ref{comdia} that this is compatible with the raising level map of section four (which is an equivalence in this case).

\section{Twisted divided powers of negative level} \label{partial}

We recast here some results from \cite{GrosLeStumQuiros19} (see also section 2 of \cite{GrosLeStumQuiros20}) and extend them to negative level. 

We let $R$ be any commutative ring.
We fix some $q \in R$ and denote by $(n)_{q} \in R$ the $q$-analog of an $n \in \mathbb Z_{\geq 0}$ (see \cite{LeStumQuiros15} for example for a presentation of the theory of $q$-analogs).
We let $A$ be any commutative $R$-algebra and fix some $x \in A$.
We also fix a prime $p$ and some $m \in \mathbb Z_{\geq 0}$.

If $y \in A$ and $n \in \mathbb Z_{\geq 0}$, we will consider the \emph{twisted powers}
\begin{equation}\label{naivpow}
\quad \xi^{(n)_{q,y}} := \prod_{i=0}^{n-1}\left(\xi + (i)_{q}y\right) \in A[\xi].
\end{equation}
They form an alternative basis for the polynomial ring $A[\xi]$ as a free $A$-module.
Note however that multiplication of twisted powers is kind of tricky because
\[
\xi^{(n_1)_{q,y}} \xi^{(n_2)_{q,y}} = \sum_{i=0}^{\min{(n_1,n_2)}} (-1)^i (i)_{q}! q^{\frac{i(i-1)}2}{n_1 \choose i}_{q}{n_2 \choose i}_{q} y^{i}\xi^{(n_1+n_2-i)_{q,y}},
\]
as shown in lemma 1.2 of \cite{GrosLeStumQuiros19}.
We will need to understand how blowing-up and frobenius act on twisted powers and we can already notice the following:

\begin{lem} If $y,z \in A$, then the blowing up
\[
A[\xi] \to A[\omega], \quad \xi \mapsto z\omega
\]
sends $\xi^{(n)_{q,zy}}$ to $z^n\omega^{(n)_{q,y}}$.
\end{lem}

\begin{proof}
We have $\prod_{i=0}^{n-1}\left(z\omega + (i)_{q}zy\right) = z^n \prod_{i=0}^{n-1}\left(\omega + (i)_{q}y\right).$
\end{proof}

Later, we will assume that $R$ is a $\delta$-ring (see section 1 of \cite{GrosLeStumQuiros20} for a short review) and in particular that $R$ is endowed with a lifting of frobenius given by $\phi(f) = f^p + p\delta(f)$.
We will also assume that $q$ has rank one which means that $\delta(q) = 0$ (so that $\phi(q) = q^p$).
We will denote\footnote{Later we will write $
A' := R {}_{{}_{\phi}\nwarrow}\!\!\widehat \otimes_{R} A$ but the topology is discrete here.} by
\[
A' := R {}_{{}_{\phi}\nwarrow}\!\!\otimes_{R} A
\]
the frobenius pullback\footnote{We always use the geometric vocabulary and say ``pullback'' instead of ``scalar extension''.} of $A$.
It is easy to see that the semilinear morphism of rings
\[
A \to A', \quad f \mapsto f' := 1 {}_{{}_{\phi}\nwarrow}\!\!\otimes f
\]
has the following effect on twisted polynomials:
\[
A[\xi] \to A'[\xi], \quad \xi^{(n)_{q,y}} \mapsto \xi^{(n)_{q^p, y'}}
\]
(with $y \in A$ and $y' = 1 {}_{{}_{\phi}\nwarrow}\!\!\otimes y$).

Back to the general situation, the next result will follow from section 2 in \cite{GrosLeStumQuiros19}:

\begin{prop}
Given $y \in A$, there exists a unique \emph{natural} multiplication on the free $A$-module $A\langle \xi \rangle_{q,y}$ with basis $\{\xi^{[n]_{q,y}}\}_{n \in \mathbb Z_{\geq 0}}$ turning the $A$-linear map
\[
A[\xi] \to A\langle \xi \rangle_{q,y}, \quad \xi^{(n)_{q,y}}\mapsto (n)_{q}! \xi^{[n]_{q,y}}
\]
into a morphism of rings.
\end{prop}

\begin{proof}
Since this is the first time that we use the term ``natural'', we should make it precise (and leave it to the reader's imagination in the future):
it means that if we are given a commutative diagram of commutative rings
\[
\xymatrix{
q_{1} \ar@{|->}[d] & R_{1} \ar[rr] \ar[d] && A_{1} \ar[d] & y_{1} \ar@{|->}[d] \\
q_{2} & R_{2} \ar[rr] && A_{2}& y_{2},
}
\]
then the diagram
\[
\xymatrix{
\xi \ar@{|->}[d] & A_{1}[\xi] \ar[rr] \ar[d] && A_{1}\langle \xi \rangle_{q_{1},y_{1}} \ar[d] & \xi^{[n]_{q_{1},y_{1}}} \ar@{|->}[d] \\
\xi & A_{2}[\xi] \ar[rr] && A_{2}\langle \xi \rangle_{q_{2},y_{2}} & \xi^{[n]_{q_{2},y_{2}}}
}
\]
is also commutative.
Let us now prove our assertion.
Existence follows from propositions 2.2 and 2.1 in \cite{GrosLeStumQuiros19}.
For uniqueness, we may first replace $A$ with some polynomial ring (in an infinite number of variables) over $\mathbb Z[q]$, then replace $\mathbb Z[q]$ with its fraction field $\mathbb Q(q)$ and finally rely on proposition 2.1 of \cite{GrosLeStumQuiros19} again.
\end{proof}

\begin{dfn} \label{zerol}
The ring $A\langle \xi \rangle_{q,y}$ is the ring of \emph{twisted divided polynomials}.
\end{dfn}

Note that, when $(n)_{q}!$ is invertible in $R$, we have
\[
\xi^{[n]_{q,y}} = \frac {\xi^{(n)_{q,y}}} {(n)_{q}!}
\]
so that $\xi^{[n]_{q,y}}$ deserves the name of a \emph{twisted divided power}.

We shall now show that blowing up and frobenius both extend to \emph{twisted} divided powers.

\begin{lem} \label{bloex}
If $y,z \in A$, then the blowing up $\xi \mapsto z\omega$ extends naturally in a unique way to a morphism of $A$-algebras
\[
A\langle\xi\rangle_{q,zy} \to A\langle\omega\rangle_{q,y}, \quad \xi^{[n]_{q,zy}} \mapsto z^n\omega^{[n]_{q,y}}.
\]
\end{lem}

\begin{proof}
First of all, the formula defines an $A$-linear map that extends the original blowing up.
In order to show that it is a morphism of rings, we may actually assume that $R = \mathbb Q(q)$ in which case the question becomes trivial.
The same argument can be used to show uniqueness.
\end{proof}

On proves exactly in the same way the following:

\begin{lem} \label{mid}
If $R$ is a $\delta$-ring with $\delta(q) = 0$ and $A'$ denotes the frobenius pullback of $A$, then there exists a unique natural semilinear morphism of rings
\[
A \langle \xi \rangle_{q,y} \to A' \langle \xi \rangle_{q^p,y'}, \quad \xi^{[n]_{q,y}} \mapsto \xi^{[n]_{q^p,y'}}
\]
(with $y \in A$ and $y' = 1 {}_{{}_{\phi}\nwarrow}\!\!\otimes y$).\qed
\end{lem}

Note that we actually obtain an isomorphism
\[
R {}_{{}_{\phi}\nwarrow}\!\!\otimes_{R} A \langle \xi \rangle_{q,y} \simeq A' \langle \xi \rangle_{q^p,y'}.
\]

Recall that we fixed some parameter $x \in A$ (that we haven't used yet).
We will then set $y := (1-q)x$ and simply write $\xi^{(n)_{q}}$ for the twisted powers so that
\[
\quad \xi^{(n)_{q}} := \prod_{i=0}^{n-1}\left(\xi + (1-q^i)x\right) \in A[\xi].
\]
In this situation, we will also denote by $A\langle \xi \rangle_{q}$ the ring of twisted divided polynomials and write $\xi^{[n]_{q}}$ for the corresponding twisted divided powers.

We can replace $q$ with $q^{p^m}$ in order to obtain a new ring $A\langle \xi \rangle_{q^{p^m}}$ but $y$ needs then to be replaced by $(1-q^{p^m})x= (p^m)_qy$ in order to fit the pattern (for fixed $x$) so that
\[
A\langle \xi \rangle_{q} := A\langle \xi \rangle_{q,y} \quad \mathrm{and} \quad A\langle \xi \rangle_{q^{p^m}} := A\langle \xi \rangle_{q^{p^m},(p^m)_{q}y}.
\]
The new ingredient in our story is the ring obtained by changing $q$ to $q^{p^m}$ but keeping the same $y$ (in which case we'd use $\omega$ rather than $\xi$ as indeterminate):

\begin{dfn} \label{pardef}
Given  $m \in \mathbb Z_{\geq 0}$, the ring of \emph{twisted divided polynomials of level $-m$ with respect to $p$} is
\[
A\langle \omega \rangle_{q(-m)} := A\langle \omega \rangle_{q^{p^m},y}
\]
with $y := (1-q)x$.
\end{dfn}

We will then write $\omega^{\{n\}_{q}} := \omega^{[n]_{q^{p^m},y}}$ and call them \emph{twisted divided powers of level $-m$}.

\begin{rmks}
\begin{enumerate}
\item
The ring of twisted divided polynomials of level $0$ is the same thing as the ring of twisted divided polynomials of definition \ref{zerol} : $A\langle \omega \rangle_{q(0)} = A\langle \xi \rangle_{q}$ (we will usually keep $\xi$ as indeterminate in this situation).
\item
If we write explicitly the multiplication on $A\langle \omega \rangle_{q(-m)}$, we get:
\begin{align*}
\omega^{\{ n_{1} \}_{q}} \omega^{\{ n_{2} \}_{q}} = \sum_{0\leq i \leq \min\{n_{1},n_{2}\}} q^{\frac{p^mi(i-1)}2}{n_{1} + n_{2} -i \choose n_{1}}_{q^{p^m}}{n_{1} \choose i}_{q^{p^m}} (q-1)^ix^i\omega^{ \{ n_{1}+n_{2}-i \}_{q} }.
\end{align*}
\end{enumerate}
\end{rmks}

We will now study blowing up and frobenius in this specific context.

\begin{prop} \label{propbl} The blowing up $\xi \mapsto (p^m)_{q} \omega$ extends uniquely to a natural morphism of $A$-algebras
\[
A\langle \xi \rangle_{q^{p^m}} \to A\langle \omega \rangle_{q(-m)}, \quad \xi^{[n]_{q^{p^m}}}\mapsto (p^m)_{q}^n \omega^{\{n\}_{q}}.
\]
\end{prop}

\begin{proof}
Apply lemma \ref{bloex} with $q$ replaced by $q^{p^m}$ and $z$ replaced by $(p^m)_{q}$.
\end{proof}

\begin{rmks}
\begin{enumerate}
\item
When $(p^m)_{q}$ (resp. and $(n)_{q^{p^m}}!$) is invertible in $R$, we may identify both rings in the proposition and we have
\[
\omega^{\{n\}_{q}} = \frac {\xi^{[n]_{q^{p^m}}}} {(p^m)_{q}^n} \quad \left(\mathrm{resp.}\ = \frac {\xi^{(n)_{q^{p^m}}}} {(p^m)_{q}^n(n)_{q^{p^m}}!}\right)
\]
In particular, we will have $\omega = \xi/(p^m)_{q}$ and this is why it is important to use different letters.
\item 
More generally, there exists a natural morphism of $A$-algebras
\[
A\langle \omega \rangle_{q^{p^r}(-m+r)}\to A\langle \omega \rangle_{q(-m)}, \quad \omega^{\{n\}_{q^{p^r}} }\mapsto (p^{m-r})_{q}^n \omega^{\{n\}_{q}}
\]
whenever $r \geq m$.
\item 
As a consequence, if $R$ is a $\delta$-ring with $\delta(q) = 0$, then there exists a unique natural semilinear morphism of rings
\[
A \langle \omega \rangle_{q(-m+1)}\to A' \langle \omega \rangle_{q(-m)}, \quad \omega^{\{n\}_{q}}\mapsto (p)_{q}^n\omega^{\{n\}_{q}}.
\]
 (using $x':= 1 {}_{{}_{\phi}\nwarrow}\!\!\otimes x \in A' := R {}_{{}_{\phi}\nwarrow}\!\!\otimes_{R} A$ as fixed parameter on the right hand side).
\end{enumerate}
\end{rmks}

Assume $R$ is a $\delta$-ring and $A$ is a $\delta$-$R$-algebra with both $q$ and $x$ of rank one.

\begin{dfn}
The \emph{symmetric $\delta$-structure} on $A[\xi]$ is defined by
\[
\delta(\xi) = \sum_{i=1}^{p-1} \frac 1p {p \choose i} x^{p-i}\xi^{i}.
\]
\end{dfn}

It is easy to check that this is the unique $\delta$-structure that extends the $\delta$-structure of $A$ and such that $x+\xi$ also has rank one.
We will also consider the corresponding \emph{relative} frobenius
\[
F : A'[\xi] = R {}_{{}_{\phi}\nwarrow}\!\!\otimes_{R} A[\xi] \to A[\xi]
\]
which is then given by $F(\xi) = (x + \xi)^p-x^p$.

We focus up to the end of the section on the case $m=1$.

\begin{thm} \label{propdivf}
Assume $R$ is a $\delta$-ring and $A$ is a $\delta$-$R$-algebra with both $q$ and $x$ of rank one.
Then the relative frobenius $F : A'[\xi] \to A[\xi]$ extends uniquely (through blowing up) to a \emph{natural} morphism
\[
[F] : A'\langle \omega \rangle_{q(-1)} \to A\langle \xi \rangle_{q}.
\]
\end{thm}

\begin{proof}
We have to prove that there exists a unique natural morphism making the diagram
\[
\xymatrix{
\xi \ar@{|->}[d] & A'[\xi] \ar[rr]^{F} \ar[d] && A[\xi] \ar[d] & \xi \ar@{|->}[d] \\
(p)_{q}\omega & A'\langle \omega \rangle_{q(-1)} \ar@{-->}[rr]^{[F]} && A\langle \xi \rangle_{q} & \xi.
}
\]
commutative.
This requires some work but this is done in section 7 of \cite{GrosLeStumQuiros19} where we showed that $[F]$ is explicitly given by
\[
\omega^{\{n\}} \mapsto \sum_{i=n}^{pn} b_{n,i}x^{pn-i}\xi^{[i]}
\]
with
\begin{equation} \label{bn}
b_{n,i} := \frac {(i)_{q}!}{(n)_{q^p}!(p)_{q}^n}a_{n,i} \in R
\end{equation}
and
\begin{equation}\label{an}
a_{n,i} := \sum_{j=0}^n (-1)^{n-j} q^{\frac {p(n-j)(n-j-1)}2} {n \choose j}_{q^p} {pj \choose i}_{q}. \qedhere
\end{equation}
\end{proof}

\begin{dfn} \label{divfrob}
The morphism $[F]$ is the \emph{divided frobenius}.
\end{dfn}

\begin{cor}
The symmetric $\delta$-structure of $A[\xi]$ extends uniquely in a natural way to $A\langle \xi \rangle_{q}$.
\end{cor}

\begin{proof}
We define $\phi$ as the composition
\[
\xymatrix{
\xi \ar@{|->}[d] & A\langle \xi \rangle_{q} \ar[d] \ar[rrd]^{\phi} && \\
(p)_{q}\omega & A'\langle \omega \rangle_{q(-1)} \ar[rr]^{[F]} && A\langle \xi \rangle_{q}
}
\]
and easily check that it solves the problem.
\end{proof}

This last result also holds for $A\langle \omega \rangle_{q(-1)}$ but we need to go back to our fancy formulas from \cite{GrosLeStumQuiros19} in order to prove it:

\begin{prop} \label{del1}
Assume $R$ is a $\delta$-ring and $A$ is a $\delta$-$R$-algebra with both $q$ and $x$ of rank one.
Then the symmetric $\delta$-structure of $A[\xi]$ extends uniquely (through blowing up) in a natural way to $A\langle \omega \rangle_{q(-1)}$.
\end{prop}

\begin{proof}
It is actually sufficient to show that there exists a unique \emph{natural} $\phi$-structure on $A\langle \omega \rangle_{q(-1)}$: we may assume that $R = \mathbb Z[q]$ and $A=R[x]$ are $p$-torsion free and the notions of $\delta$ and $\phi$-structures are then equivalent.
In order to avoid confusions, let us denote by $\phi_{A} : A \to A$ the frobenius of $A$.
It follows from proposition 7.5 (and lemma 7.8 which shows that the sum actually starts at $i=n$) of \cite{GrosLeStumQuiros19} that the (absolute) frobenius on $A[\xi]$ is the $\phi_{A}$-linear morphism given by
\[
\phi : A[\xi] \to A[\xi], \quad \xi^{(n)_{q^p}} \mapsto \sum_{i=n}^{pn} \phi(a_{n,i})x^{pn-i}\xi^{(i)_{q^p}}
\]
where the $a_{n,i}$ are defined in \eqref{an}.
We consider now the natural $\phi_{A}$-linear morphism of $A$-modules
\[
\phi : A\langle \omega \rangle_{q(-1)} \to A\langle \omega \rangle_{q(-1)}, \quad \omega^{\{n\}_{q}} \mapsto \sum_{i=n}^{pn} (p)_{q}^i\phi(b_{n,i})x^{pn-i}\omega^{\{i\}_{q}}
\]
where the $b_{n,i}$ are defined in \eqref{bn}.
We want to show that this is a morphism of rings that reduces to the frobenius modulo $p$.
Actually, since the blowing up map
\[
A[\xi] \to A\langle \omega \rangle_{q(-1)}, \quad \xi^{(n)_{q^p}}\mapsto (n)_{q^p}! (p)_{q}^n \omega^{\{n\}_{q}}
\]
becomes an isomorphism when all $q$-analogs are invertible, it is sufficient to prove that the diagram
\[
\xymatrix{
\xi^{(n)_{q^p}} \ar@{|->}[d] & A[\xi] \ar[rr]^{\phi} \ar[d] && A[\xi] \ar[d] & \xi^{(n)_{q^p}} \ar@{|->}[d] \\
(n)_{q^p}! (p)_{q}^n \omega^{\{n\}_{q}} & A\langle \omega \rangle_{q(-1)} \ar[rr]^{\phi} && A\langle \omega \rangle_{q(-1)} & (n)_{q^p}! (p)_{q}^n \omega^{\{n\}_{q}}
}
\]
is commutative.
We compute the image of $\xi^{(n)_{q^p}}$ along both paths.
On one hand, we find
\[
\sum_{i=n}^{pn} (i)_{q^p}! (p)_{q}^i \phi(a_{n,i})x^{pn-i}\omega^{\{i\}_{q}}
\]
and on the other hand
\[
\phi\left((n)_{q^p}! (p)_{q}^n\right) \sum_{i=n}^{pn} (p)_{q}^i\phi(b_{n,i})x^{pn-i}\omega^{\{i\}_{q}}.
\]
In order to conclude, it is therefore sufficient to recall that $\phi((i)_{q}!) = (i)_{q^p}!$ and formula \eqref{bn} then provides an equality
\[
(i)_{q^p}! \phi(a_{n,i}) = \phi\left((n)_{q^p}! (p)_{q}^n\right)\phi(b_{n,i}).\qedhere
\]
\end{proof}

\begin{xmp}
Assume $p=2$.
Then, we have
\[
\phi(\xi) = (1+q)\xi^{[2]}+ (1+q)x\xi \quad \mathrm{and} \quad \phi(\omega) = (1+q)^2\omega^{\{2\}} + (1+q)x\omega.
\]
In particular, we see that there is no \emph{natural} $\delta$-structure (i.e. compatible with blowing up) on the polynomial ring $A[\omega]$ in general:
we do have $\phi(\xi) \in A[\xi]$ but $\phi(\omega) \notin A[\omega]$.
\end{xmp}

\begin{rmk}
The diagram
\[
\xymatrix{
A'\langle \omega \rangle_{q(-1)} \ar[rr]^{[F]} \ar[rrd]^{\phi} && A\langle \xi \rangle_{q} \ar[d] & \xi \ar@{|->}[d] \\
&& A'\langle \omega \rangle_{q(-1)} & (p)_{q}\omega
}
\]
is commutative and may be used to define the $\delta$-structure of $A\langle \omega \rangle_{q(-1)}$ when $A$ itself is a frobenius pullback (which will be the case in practice).
\end{rmk}

\section{Twisted coordinate} \label{coord}

This section is completely independent of the previous one.
We recall some notions introduced in \cite{LeStumQuiros18} and give an important example of a situation where these notions apply.
However, we work on adic rings as in \cite{GrosLeStumQuiros20}, and not merely on usual rings.

We assume that $R$ is an adic ring (not necessarily complete) and that $A$ is a \emph{complete} adic $R$-algebra (not necessarily $R$-adic at this point).
Adic rings are \emph{always} assumed to admit a finitely generated ideal of definition.
Unless otherwise specified, completion is always meant relative to the adic topology.
We will not need it here but we could even assume that $R$ and $A$ are Huber rings as long as $P_{A/R} := A \otimes_{R} A$ is also a Huber ring (multiplication on $P_{A/R}$ need not be continuous in general).

We denote by
\[
p_{1} : A \to P_{A/R}, f \mapsto f \otimes 1 \quad \mathrm{and} \quad p_{2} : A \to P_{A/R}, f \mapsto 1 \otimes f
\]
the ``projections'' and by
\[
\Delta : P_{A/R} \to A, f \otimes g \mapsto fg
\]
the ``diagonal map'' (we use the - contravariant - geometric vocabulary and notations).
Unless otherwise specified, we will always consider $P_{A/R}$ as an $A$-module through $p_{1}$ and call $p_{2}$ the ``Taylor map''.

Assume now that $A$ is a \emph{twisted $R$-algebra}, which simply means that the $R$-algebra $A$ is endowed with a \emph{continuous} endomorphism $\sigma$.
We extend $\sigma$ to $P_{A/R}$ in an asymmetric way by the formula $\sigma(f \otimes g) = \sigma(f) \otimes g$.
We let $I_{A/R}$ be the kernel of $\Delta$ and define the \emph{twisted powers} and \emph{twisted principal parts of order $n$}:
\[
I_{A/R}^{(n+1)_{\sigma}} := I_{A/R}\sigma(I_{A/R}) \cdots \sigma^n(I_{A/R})
\quad \mathrm{and} \quad
P_{A/R,(n)_{\sigma}} := \widehat P_{A/R}/\overline{I_{A/R}^{(n+1)_{\sigma}}}
\]
(where the hat (resp. the bar) indicates the completion (resp. the closure)).
\begin{dfn}
Assume $A$ is a \emph{twisted $R$-algebra} and let $x \in A$.
Then $x$ is a \emph{$\sigma$-coordinate} if the canonical map
\[
A[\xi]_{\leq n} \to P_{A/R,(n)_{\sigma}}, \quad \xi \mapsto \overline{1 \otimes x - x \otimes 1}
\]
is bijective\footnote{This is the $\sigma$-analog of \emph{differential smoothness} (of relative dimension one).} for all $n \in \mathbb Z_{\geq 0}$.
\end{dfn}

It is actually convenient to extend $\sigma$ to the polynomial ring $A[\xi]$ by requiring that $\sigma(x+\xi) = x + \xi$ and to introduce the \emph{twisted powers} $\xi^{(n+1)_{\sigma}} := \xi\sigma(\xi)\ldots \sigma^n(\xi) \in A[\xi]$.
One easily sees that $x$ is a $\sigma$-coordinate if and only if
\[
A[\xi]/(\xi^{(n+1)_{\sigma}}) \simeq P_{A/R,(n)_{\sigma}}
\]
(isomorphism of $A$-algebras).
In general, we also set
\[
\Omega_{A/R,\sigma} := \overline I_{A/R}/\overline I_{A/R}^{(2)_{\sigma}}
\]
and we denote by $\mathrm d_{\sigma} : A \to \Omega_{A/R,\sigma}$ the map induced by $p_{2} - p_{1}$.
This $A$-module represents the \emph{$\sigma$-derivations} $D : A \to M$, i.e. the $R$-linear maps satisfying the twisted Leibniz rule
\[
\forall f, g \in A, \quad D(fg) = D(f)g + \sigma(f)D(g).
\]
When $x$ is a $\sigma$-coordinate, $\Omega_{A/R,\sigma}$ is a free module on the generator $\mathrm d_{\sigma}x$ and we denote by $\partial_{\sigma}$ the corresponding derivation of $A$ (so that $\mathrm d_{\sigma} f = \partial_{\sigma}(f)\mathrm d_{\sigma} x$).

In order to make clear the relation with the notions introduced in section \ref{partial}, note that if $x$ is a $\sigma$-coordinate such that $\sigma(x) = qx$ with $q \in R$, then $\xi^{(n+1)_{q}} = \xi^{(n+1)_{\sigma}}$.
In this situation, we will use $q$ rather than $\sigma$ as index or prefix in all our notations so that we will write $I_{A/R}^{(n+1)_{q}}$, $P_{A/R,(n)_{q}}$, $\Omega_{A/R,q}$, $\mathrm d_{q}$ or $\partial_{q}$, for example, and say $q$-coordinate or $q$-derivation (even if everything actually depends on $\sigma$).

Let us give an important example where there exists a $q$-coordinate (we already used these results without proofs in \cite{GrosLeStumQuiros20}).

\begin{dfn} \label{topet}
A \emph{topologically \'etale coordinate} $x$ on $A$ with respect to $R$ is the image of an indeterminate $X$ under some topologically \'etale (that is, formally \'etale and topologically finitely presented) map $R[X] \to A$.
\end{dfn}

\begin{prop} \label{liftsig}
If $q-1$ is topologically nilpotent and $x$ is a topologically \'etale coordinate on $A$, then there exists a unique continuous endomorphism $\sigma$ of $A$ such that $\sigma(x) = qx$.
Moreover, $x$ is a $q$-coordinate on $A$.
\end{prop}

\begin{proof}
Since $A$ is complete adic and $q-1$ is topologically nilpotent, $A$ is complete for the $(q-1)$-adic topology.
We may therefore assume that $q-1$ is nilpotent.
The existence and uniqueness of $\sigma$ are trivially true when $A = R[X]$ and $x = X$.
In general, they follow from the commutativity of the diagram
\[
\xymatrix{
R[X] \ar[r]^{\sigma} \ar[d] & R[X] \ar[r] & A. \ar@{->>}[d] \\
A \ar[rr] \ar@{-->}[rru] && A/(q-1)
}
\]
since $R[X] \to A$ is formally \'etale.
Now, since $x$ is a topologically \'etale coordinate on $A$ over $R$, it is also a topologically \'etale coordinate on $A/(q-1)$ over $R/(q-1)$.
Thus, modulo $q-1$, the canonical map 
\begin{equation} \label{canm}
A[\xi]/\xi^{(n+1)} \to P_{A/R,(n)_{q}}
\end{equation}
reduces to an isomorphism
\[
(A/(q-1))[\xi]/\xi^{n+1} \simeq P_{A/(q-1),n}.
\]
Surjectivity of \eqref{canm} therefore follows from Nakayama's lemma.
Moreover, we can build a section using the commutative diagram
\[
\xymatrix{
R[X,\xi]/\xi^{(n+1)} \ar[rr] \ar[d] && A[\xi]/\xi^{(n+1)} \ar@{->>}[d] \\
P_{A,(n)_{q}} \ar[r] \ar@{-->}[rru] & P_{A/(q-1),n} & (A/(q-1))[\xi]/\xi^{n+1} \ar[l]_-\simeq 
}
\]
and obtain injectivity.
\end{proof}

One may also produce a frobenius on $A$ using the same methods:

\begin{prop} \label{frolif}
Assume $p$ is a topologically nilpotent prime in $A$ and $R$ is a $\phi$-ring.
If $x$ is a topologically \'etale coordinate on $A$, then there exists a unique structure of $\phi$-$R$-algebra on $A$ with $x$ of rank one (i.e. $\phi(x) = x^p$). 
Moreover, the relative frobenius
\[
F: A' := R {}_{{}_{\phi}\nwarrow}\!\!\otimes_{R} A \to A
\]
is free of rank $p$.
\end{prop}

\begin{proof}
We proceed exactly as above and it is sufficient to prove existence and uniqueness of the \emph{relative} frobenius $F$.
First, we reduce to the case where $p$ is nilpotent in $A$.
In the case $A = R[X]$ and $x = X$, there is nothing to do.
In general, existence and uniqueness follow from the commutativity of the diagram
\[
\xymatrix{
R[X] \ar[rr]^{F} \ar[d] && R[X] \ar[r] & A \ar@{->>}[d] \\
A' \ar@{->>}[rr] \ar@{-->}[rrru] && A'/p \ar[r]^{F} & A/p ,
}
\]
in which the last map is the \emph{relative} frobenius of $A/p$.
It only remains to verify that $F:A' \to A$ is free of rank $p$: actually, there exists an isomorphism $A'[T]/(T^p -1 \otimes x) \simeq A$ which is obtained from the analog map when $A = R[X]$ after pulling back along $R[X] \to A'$.
\end{proof}

We can improve a little bit on the previous result (see also \cite{BhattScholze19}, lemma 2.18):

\begin{prop} \label{frolifp}
Assume $p$ is a topologically nilpotent prime in $A$ and $R$ is a $\delta$-ring.
If $x$ is a topologically \'etale coordinate on $A$, then there exists a unique structure of $\delta$-$R$-algebra on $A$ with $x$ of rank one (i.e. $\delta(x) = 0$).
\end{prop}

\begin{proof}
There exists a unique structure of $\delta$-$R$-algebra on $R[X]$ such that $\delta(X) = 0$.
Now, by definition (see section 1 of \cite{GrosLeStumQuiros20}), a $\delta$-structure on $A$ is a section of the projection $\mathrm W_{1}(A) \to A$ (where $W_{1}$ denote the Witt vectors of length two).
Note that the kernel $V_{1}$ of this projection satisfies $V_{1}^2 \subset pV_{1}$.
By functoriality, there exists a commutative diagram
\[
\xymatrix{
\mathrm W_{1}(R[X]) \ar[r] \ar[d]& R[X] \ar@/_.5cm/[l] \ar[d] \\
\mathrm W_{1}(A) \ar[r] & A.
}
\]
where the backward upper-map is the above trivial $\delta$-structure of $R[X]$.
Since $V_{1}$ is nilpotent modulo $p$ and $p$ is topologically nilpotent, the composite map
\[
R[X] \to \mathrm W_{1}(R\{X\}) \to W_{1}(A)
\]
extends uniquely to $A$ and defines a $\delta$-structure on $A$ with $\delta(x) = 0$.
\end{proof}

\section{Twisted calculus of negative level} \label{neg}

We keep the same hypotheses as before: $R$ is an adic ring and $A$ is a complete adic $R$-algebra, but we also fix\footnote{Actually, the whole story will only depend on the integer $k:=p^m$ but the terminology and the notations involve $m$.}, as in section \ref{partial}, a prime $p$ and some $m \in \mathbb Z_{\geq 0}$.
We assume that $A$ is a twisted $R$-algebra and that $x \in A$ is a $q^{p^m}$-coordinate (we will denote by $\sigma^{p^m}$ the corresponding endomorphism of $A$ even if it needs not be a power).

\begin{dfn}
A \emph{twisted connection of level $-m$} on an $A$-module $M$ is an $R$-linear map $\nabla : M \to M \otimes \Omega_{A/R,q^{p^m}}$ such that
\[
\forall f \in A, \forall s \in M, \quad \nabla(fs) = (p^m)_{q}\ s \otimes \mathrm d_{q^{p^m}}f + \sigma^{p^m}(f)\nabla(s).
\]
\end{dfn}

Intuitively, this is a twisted connection with a vertical pole: a twisted connection on $M[1/(p^m)_{q}]$ which is defined over $A$.

We will denote\footnote{Be careful that our notation is different from Shiho's \cite{Shiho15}. He use $\mathrm{MIC}^{(m)}$ for connections of level $-m$ but we prefer to keep $\mathrm{MIC}^{(m)}$ for connections of positive level.} by $\mathrm{MIC}_{q}^{(-m)}(A/R)$ the category of $A$-modules endowed with a twisted connection of level $-m$.
The \emph{twisted de Rham cohomology of level $-m$} of $M$ is defined as
\[
H^0_{\mathrm{dR},q(-m)}(M) = \ker \nabla \quad \mathrm{and} \quad H^1_{\mathrm{dR},q(-m)}(M) = \mathrm{coker}\; \nabla
\]
or in the derived category as
\[
\mathrm R\Gamma_{\mathrm{dR},q(-m)}(M) = [M \overset \nabla \to M\otimes \Omega_{A/R,q^{p^m}}].
\]

\begin{xmps}
\begin{enumerate}
\item When $m=0$, we fall back onto the notion of a twisted connection from our previous articles: $\mathrm{MIC}_{q}^{(0)}(A/R) = \mathrm{MIC}_{q}(A/R)$.
\item When $q=1$, we remove the word ``twisted'' and we obtain a \emph{$p^m$-connection} in the sense of Deligne or Simpson
as in definition 1.1 of \cite{Shiho15}.
\item When $(p^m)_{q} = 0$ and $\sigma^{p^m} = \mathrm{Id}$ (which is then automatic in practice), a twisted connection of level $-m$ is the same thing as a \emph{Higgs field} and does not depend on $q$, $p$ or $m$ anymore.
This condition obviously holds when $p=0$ but this also happens for example in the important case when $q$ is a $p^m$th root of unity.
\item The notion of a twisted connection of negative level is also related to the notion of a \emph{differential operator of finite radius} that was studied in \cite{LeStumQuiros17} as well as, in the untwisted case, to the notion of a \emph{congruence level} introduced in \cite{HuygheSchmidtStrauch17*}.
See also corollary 1.3.15 in \cite{Kisin06}, where Mark Kisin associates a $\lambda$-connection to a weakly admissible module.
\end{enumerate}
\end{xmps}

\begin{dfn}
A \emph{twisted derivation of level $-m$} on an $A$-module $M$ is an $R$-linear map $\theta : M \to M$ such that
\[
\forall f \in A, \forall s \in M, \quad \theta(fs) = (p^m)_{q} \partial_{q^{p^m}}(f)s + \sigma^{p^m}(f)\theta(s).
\]
\end{dfn}

Of course, twisted connections and twisted derivations of level $-m$ correspond bijectively via
\[
\nabla(s) = \theta(s)\mathrm d_{q^{p^m}}x,
\]
but the later notion is often more convenient for computations.

\begin{dfn}
The \emph{ring of twisted differential operators of level $-m$ on $A/R$} is the free $A$-module $\mathrm D_{A/R,q}^{(-m)}$ on the generators $\partial_{q}^{\langle n \rangle}$ for $n \in \mathbb Z_{\geq 0}$ with the commutation rules
\[
\forall f \in A, \quad \partial_{q}^{\langle 1 \rangle} \circ f = (p^m)_{q} \partial_{q^{p^m}}(f) + \sigma^{p^m}(f)\partial_{q}^{\langle 1 \rangle}
\]
and
\[
\forall n_{1}, n_{2} \in \mathbb Z_{\geq 0}, \quad \partial_{q}^{\langle n_{1} \rangle} \circ \partial_{q}^{\langle n_{2} \rangle} = \partial_{q}^{\langle n_{1} + n_{2} \rangle}.
\]
\end{dfn}

\begin{rmks}
\begin{enumerate}
\item
If we assume that $x$ is also a $q$-coordinate and if we denote by $\mathrm D_{A/R,q}$ the ring of twisted differential operators that we already considered in our previous articles, then there exists a natural $A$-linear morphism of rings
\begin{align} \label{incl}
\mathrm D_{A/R,q}^{(-m)} \to \mathrm D_{A/R,q}, \quad \partial_{q}^{\langle n\rangle} \mapsto (p^m)_{q}^n\partial_{q^{p^m}}^n.
\end{align}
More generally, we have
\[
\mathrm D_{A/R,q}^{(-m)} \to \mathrm D_{A/R,q^{p^r}}^{(-m+r)}, \quad \partial_{q}^{\langle n\rangle} \mapsto (p^r)_{q}^n\partial_{q^{p^r}}^{\langle n\rangle}
\]
for any $r \leq m$ if $x$ is also a $q^{p^{m-r}}$-coordinate.
\item
Giving a twisted connection (or a twisted derivation) of level $-m$ on an $A$-module $M$ is equivalent to giving the structure of a $\mathrm D_{A/R,q}^{(-m)}$-module.
For the trivial twisted connection of level $-m$ on $A$, we can check that
\[
 \partial_{q}^{\langle n \rangle}(f) = (p^m)_{q}^n \partial_{q^{p^m}}^{n}(f).
\]
\item It is not difficult to see that
\[
\mathrm R\Gamma_{dR,q(-m)}(M) = \mathrm {RHom}_{\mathrm D_{A/R,q}^{(-m)}}(A,M).
\]
\end{enumerate}
\end{rmks}

\begin{xmps}
\begin{enumerate}
\item In the case $m=0$, we fall back onto the ring $\mathrm D_{A/R,q}$ already considered in our previous articles.
\item
In the case $q=1$, then $\mathrm D_{A/R,q}^{(-m)}$ is the same thing as the ring $\mathrm D_{A/R}^{(-m)}$ of differential operators of level $-m$ introduced in definition 2.1 of \cite{Shiho15} (see also \cite{HuygheSchmidtStrauch17*}).
\item When $A$ is $(p)_{q}$-torsion-free and $x$ is also a $q$-coordinate, the map \eqref{incl} is injective, and we could then as well define $\mathrm D_{A/R,q}^{(-m)}$ as the $A$-subalgebra of $\mathrm D_{A/R,q^{p^m}}$ generated by $\partial_{q}^{\langle 1\rangle} := (p^m)_{q}\partial_{q^{p^m}}$.
\end{enumerate}
\end{xmps}

Recall that we introduced in definition \ref{pardef} the ring $A\langle \omega \rangle_{q(-m)}$ of twisted divided polynomials of level $-m$.
We will denote by $I_{A/R}^{\{n+1\}_{q}}$ the free $A$-module generated by $\omega^{\{ n' \}_{q}}$ for $n' > n$.
This defines an ideal filtration on $A\langle \omega \rangle_{q(-m)}$.

\begin{lem}
\begin{enumerate}
\item
The blowing up $\xi \mapsto (p^m)_{q}\omega$ induces a canonical morphism of $A$-algebras
\[
A[\xi]/(\xi^{(n+1)_{q^{p^m}}}) \to A\langle\omega\rangle_{q(-m)}/I_{A/R}^{\{n+1\}_{q}}.
\]
\item 
If $(p^m)_{q}^n (n)_{q^{p^m}}! \to 0$ in $A$, then the blowing up provides a canonical morphism of $A$-algebras
\[
A[[\xi]]_{q^{p^m}} := \varprojlim A[\xi]/(\xi^{(n+1)_{q^{p^m}}}) \to \widehat {A\langle\omega\rangle}_{q(-m)}, \quad \xi^{(n)_{q^{p^m}}} \mapsto (p^m)_{q}^n (n)_{q^{p^m}}! \omega^{\{n\}_{q}}
\]
(recall that completion is meant with respect to the adic topology of $A$).
\end{enumerate}
\end{lem}

\begin{proof}
The first assertion follows from the fact that $(p^m)_{q}^{n} (n)_{q^{p^m}}! \omega^{\{n\}_{q}} \in I_{A/R}^{\{n\}_{q}}$ and the second one is obtained by taking the limit on both sides.
More precisely, our hypothesis implies the existence of an ordered function $n \mapsto r(n)$ and a canonical morphism
\[
A[\xi]/(\xi^{(n+1)_{q^{p^m}}}) \to A\langle\omega\rangle_{q(-m)}/\mathfrak a^{r(n)}A\langle\omega\rangle_{q(-m)}
\]
if $\mathfrak a$ is an ideal of definition in $A$.
\end{proof}

\begin{dfn} \label{deftayl}
\begin{enumerate}
\item
The \emph{twisted Taylor map level $-m$ and order $n$} of $A/R$ is the composite
\[
\xymatrix{
\theta_{n} : A \ar[r]^-{p_{2}} & P_{A/R} \ar[r] & \widehat P_{A/R}/\widehat I_{A/R}^{(n+1)_{q^{p^m}}} \ar[d]^-\simeq \\ && A[\xi]/(\xi^{(n+1)_{q^{p^m}}}) \ar[r] & {A\langle\omega\rangle}_{q(-m)}/I_{A/R}^{\{n+1\}_{q}}.
}
\]
\item
If $(p^m)_{q}^n (n)_{q^{p^m}}! \to 0$ in $A$, then the \emph{twisted Taylor map of level $-m$} of $A/R$ is the composite map
\[
\xymatrix@R=0cm{
\theta : A \ar[r]^-{p_{2}} & P_{A/R} \ar[r] & \varprojlim \widehat P_{A/R}/\widehat I_{A/R}^{(n+1)_{q^{p^m}}} \ar[r]^-\simeq & A[[\xi]]_{q^{p^m}} \ar[r] & \widehat {A\langle\omega\rangle}_{q(-m)}.
}
\]
\end{enumerate}
\end{dfn}

\begin{rmks}
\begin{enumerate}
\item
The twisted Taylor map of level $-m$ is a morphism of $R$-algebras that endows ${A\langle\omega\rangle}_{q(-m)}/I_{A/R}^{\{n+1\}_{q}}$ (resp. $\widehat {A\langle\omega\rangle}_{q(-m)}$) with a second structure of $A$-module commonly called the ``right'' structure (the usual one being the ``left'' structure).
\item
It will follow from proposition \ref{dual} below that the twisted Taylor map of level $-m$ is explicitly given by
\[
f \mapsto \sum \partial_{q}^{\langle i \rangle}(f)\omega^{\{i\}_{q}}.
\]
This actually provides an alternative definition.
\end{enumerate}
\end{rmks}

\begin{dfn}
A \emph{twisted differential operator of level $-m$ and order at most $n$} is an $A$-linear map
\[
u : A\langle \omega \rangle_{q(-m)}/I_{A/R}^{\{n+1\}_{q}} \otimes'_{A} M \to N
\]
(in which the $\otimes'$ indicates that we are using the ``right'' structure given by the twisted Taylor map of level $-m$ on the left hand side).
\end{dfn}

One can compose twisted differential operators of level $-m$ of different orders in the usual way using the diagonal maps
\begin{equation} \label{diagon2}
\xymatrix@R=0cm{
A\langle\omega\rangle_{q(-m)}/I_{A/R}^{\{n_{1} + n_{2} +1\}_{q}} \ar[rr]^-{\Delta_{n_{1},n_{2}}} && A\langle\omega\rangle_{q(-m)}/I_{A/R}^{\{n_{1}+1\}_{q}} \otimes'_{A}A\langle\omega\rangle_{q(-m)}/I_{A/R}^{\{n_{2}+1\}_{q}} \\
\omega^{\{i\}_{q}} \ar@{|->}[rr] && \sum_{i_{1}+i_{2}=i} \omega^{\{i_{1}\}_{q}} \otimes \omega^{\{i_{2}\}_{q}}.
}
\end{equation}
More precisely, the composition of two differential operators $u_{1}$ and $u_{2}$ of order at most $n_{1}$ and $n_{2}$ respectively is the differential operator of order at most $n_{1} + n_{2}$ defined by
\[
u_{1}u_{2}:= u_{1} \circ (\mathrm{Id} \otimes' u_{2}) \circ (\Delta_{n_{1},n_{2}} \otimes' \mathrm{Id}).
\]

\begin{prop} \label{dual}
The twisted differential operators of level $-m$ of all orders from $A$ to itself form a ring which is isomorphic to $D_{A/R,q}^{(-m)}$.
More precisely, the basis $\{\omega^{\{n\}_{q}}\}_{n \in \mathbb Z_{\geq 0}}$ of $A\langle \omega \rangle_{q(-m)}$ is ``topologically'' dual to the basis $\{\partial_{q}^{\langle n \rangle }\}_{n\in \mathbb Z_{\geq 0}}$ of $D_{A/R,q}^{(-m)}$.
\end{prop}

\begin{proof}
The point is to check that the bilinear map
\[
\mathrm D_{A/R,q}^{(-m)} \times A\langle \omega \rangle_{q(-m)}/I_{A/R}^{\{n+1\}_{q}} \to A, \quad (\partial_{q}^{\langle n \rangle}, \omega^{\{k\}}_{q}) \mapsto \left\{\begin{array}{l} 1 \ \mathrm{if}\ n=k \\ 0 \ \mathrm{otherwise}\end{array}\right.
\]
provides an isomorphism of rings
\[
\mathrm D_{A/R,q}^{(-m)} \simeq \varinjlim_{n} \mathrm{Hom}_{A}\left(A\langle \omega \rangle_{q(-m)}/I_{A/R}^{\{n+1\}_{q}}, A\right).
\]
Details are left to the reader.
\end{proof}

\begin{dfn}
A \emph{twisted Taylor structure of level $-m$} on an $A$-module $M$ is a compatible family of $A$-linear maps
\[
\theta_{n} : M \to M \otimes_{A} A\langle \omega \rangle_{q(-m)}/I_{A/R}^{\{n+1\}_{q}}
\]
such that
\[
\forall n_{1}, n_{2} \in \mathbb Z_{\geq 0}, \quad (\theta_{n_{1}} \otimes \mathrm{Id}_{A\langle \omega \rangle_{q(-m)}/I^{\{n_{2}+1\}_{q}}}) \circ \theta_{n_{2}} = (\mathrm{Id}_{M} \otimes \Delta_{n_{1},n_{2}}) \circ \theta_{n_{1}+n_{2}}.
\]
\end{dfn}

\begin{rmks}
\begin{enumerate}
\item
A twisted Taylor structure of level $-m$ is equivalent to a $\mathrm D_{A/R,q}^{(-m)}$-module structure via the formula
\[
\theta_{n}(s) = \sum_{i=0}^n \partial_{q}^{\langle i \rangle }s \otimes \omega^{\{i\}_{q}}.
\]
This is therefore also equivalent to a twisted connection (or derivation) of level $-m$.
\item One may define the \v Cech-Alexander cohomology of an $A$-module with a twisted Taylor structure of level $-m$ and show that this is isomorphic to de Rham cohomology.
\item
As in definition \ref{hyperq} below,  when $(p)_{q}^n (n)_{q^p}! \to 0$, there exists an ``hyper'' equivalent to the notion of a twisted Taylor structure of level $-m$ with $A\langle \omega \rangle_{q(-m)}/I_{A/R}^{\{n+1\}_{q}}$ replaced with of $\widehat {A\langle\omega\rangle}_{q(-m)}$.
\end{enumerate}
\end{rmks}

There remains one more notion to investigate (we denote by
\[
\Delta_{n} : A\langle \omega \rangle_{q(-m)}/I^{\{n+1\}_{q}} \to A
\]
the canonical map sending $\omega$ to $0$).

\begin{dfn} \label{hyperq}
\begin{enumerate}
\item 
A \emph{twisted stratification of level $-m$} on an $A$-module $M$ is a compatible family of $A\langle \omega \rangle_{q(-m)}/I^{\{n+1\}_{q}}$-linear isomorphisms
\[
\epsilon_{n} : A\langle \omega \rangle_{q(-m)}/I^{\{n+1\}_{q}} \otimes_{A}' M \simeq M \otimes_{A} A\langle \omega \rangle_{q(-m)}/I^{\{n+1\}_{q}}
\]
satisfying the normalization and cocycle conditions
\[
\Delta_{n}^*(\epsilon_{n}) = \mathrm{Id}_{M} \quad \mathrm{and} \quad \Delta_{n,n}^*(\epsilon_{n}) = p_{1}^*(\epsilon_{n}) \circ p_{2}^*(\epsilon_{n}).
\]
\item If $(p)_{q}^n (n)_{q^p}! \to 0$, then a \emph{twisted  hyper-stratification of level $-m$} on an $A$-module $M$ is an $\widehat{A\langle \omega \rangle}_{q(-m)}$-linear isomorphism
\[
\epsilon : \widehat{A\langle \omega \rangle}_{q(-m)} \otimes_{A}' M \simeq M \otimes_{A} \widehat{A\langle \omega \rangle}_{q(-m)}
\]
satisfying the normalization and cocycle conditions. 
\end{enumerate}
\end{dfn}

We will denote by $\mathrm{Strat}_{q}^{(-m)}(A/R)$ (resp. $\widehat{\mathrm{Strat}}_{q}^{(-m)}(A/R)$) the category of twisted stratified (resp.  hyper-stratified) modules of level $-m$.

\begin{rmks}
\begin{enumerate}
\item
The following are equivalent:
\begin{itemize}
\item a twisted connection (or derivation) of level $m$,
\item a structure of a $D_{A/R,q}^{(-m)}$-module,
\item a twisted Taylor structure of level $m$,
\item a twisted stratification of level $m$.
\end{itemize}
In particular, there exists an isomorphism of categories
\[
\mathrm{MIC}_{q}^{(-m)}(A/R) \simeq \mathrm{Strat}_{q}^{(-m)}(A/R).
\]
\item
Any twisted  hyper-stratification of level $-m$  induces a twisted stratification of level $-m$ and we get a fully faithful functor
\[
\widehat{\mathrm{Strat}}_{q}^{(-m)}(A/R) \to \mathrm{Strat}_{q}^{(-m)}(A/R) \simeq \mathrm{MIC}_{q}^{(-m)}(A/R).
\]
\end{enumerate}
\end{rmks}

\begin{prop}
When $M$ is finitely presented, a hyper-stratification of level $-m$ is equivalent to a \emph{topologically quasi-nilpotent} twisted connection of level $-m$, meaning that
\[
\forall s \in M, \quad \partial^{\langle k \rangle}(s) \to 0.
\]
\end{prop}

\begin{proof}
Standard (see the proof of proposition 6.3 in \cite{GrosLeStumQuiros20}, for example).
\end{proof}

\section{Level raising and Frobenius descent}

We fix a prime $p$ and a non-negative integer $m$.
All commutative rings are assumed to be $\mathbb Z[q]_{(p,q-1)}$-algebras and are endowed with their $(p, q-1)$-adic topology.
Let $R$ be a $(p)_{q}$-torsion free $\delta$-ring such that $q := q1_{R}$ has rank one and $A$ a complete $R$-algebra with fixed topologically \'etale coordinate $x$ (see definition \ref{topet}).

We saw in propositions \ref{liftsig} and \ref{frolifp} that, in this situation, there exists on $A$ a unique structure of twisted $R$-algebra such that $x$ is a $q$-coordinate and a unique $\delta$-structure such that $x$ has rank one.
Note that $x$ is also a $q^{p^m}$-coordinate, i.e. a coordinate with respect to $\sigma^{p^m}$ where $\sigma$ denotes the endomorphism of $A$.
Also, the construction applies as well to $A' := R {}_{{}_{\phi}\nwarrow}\!\!\widehat\otimes_{R} A$ with $x$ replaced by $x':=1 {}_{{}_{\phi}\nwarrow}\!\!\widehat\otimes x$.
Finally, recall from proposition \ref{frolif} that the relative frobenius $F : A' \to A$ is free on the generators $1, x, \ldots, x^{p-1}$.

We will need below the following commutation rules:

\begin{lem} \label{commute} 
\begin{enumerate}
\item
$F\circ \sigma^{p^{m}} = \sigma^{p^{m-1}} \circ F$
\item
$(p)_{q^{p^{m-1}}}x^{p-1} F \circ \partial_{q^{p^{m}}} = \partial_{q^{p^{m-1}}} \circ F$.
\end{enumerate}
\end{lem}

\begin{proof}
\begin{enumerate}
\item
Since both $p$ and $q-1$ are topologically nilpotent and $x'$ is a topologically \'etale coordinate on $A'$, there exists a unique morphism $A' \to A$ sending $x'$ to $q^{p^m}x^p$.
Now, we compute
\[
(F \circ \sigma^{p^m})(x') =F(q^{p^m}x') = q^{p^{m}}x^p
\]
and
\[
(\sigma^{p^{m-1}} \circ F)(x') = \sigma^{p^{m-1}}(x^p) = (\sigma^{p^{m-1}}(x))^p =  (q^{p^{m-1}}x)^p  = q^{p^{m}}x^p.
\]
\item
We first show that both maps are actually $q^{p^{m}}$-derivations from $A'$ to $A$ (seen as an $A'$-module via $F$).
In other words, they must both satisfy
\[
\forall f,g \in A', \quad D(fg) = F(f) D(g)+ F(\sigma^{p^m}(g))D(f).
\]
For the left hand side, this is automatic because we are composing a $q^{p^m}$-derivation of $A'$ with an $A'$-linear map.
For the right hand side, this follows from the first part of the lemma because
\begin{align*}
(\partial_{q^{p^{m-1}}} \circ F)(fg) & = \partial_{q^{p^{m-1}}}(F(f)F(g))\\
&= F(f)\partial_{q^{p^{m-1}}}(F(g)) + \sigma^{p^{m-1}}(F(g))\partial_{q^{p^{m-1}}}(F(f)) 
\\& = F(f)(\partial_{q^{p^{m-1}}} \circ F)(g) + F(\sigma^{p^{m}}(g))(\partial_{q^{p^{m-1}}} \circ F)(f) .
\end{align*}
Since $\Omega_{A'/R,q^{p^m}}$ is free on $\mathrm d_{q^{p^m}}x'$, a $q^{p^m}$-derivation of $A'$ is determined by its value on $x'$.
Now, we compute
\[
(p)_{q^{p^{m-1}}}x^{p-1} (F \circ \partial_{q^{p^{m}}})(x') = (p)_{q^{p^{m-1}}} x^{p-1}
\]
and
\[
 (\partial_{q^{p^{m-1}}} \circ F)(x') = \partial_{q^{p^{m-1}}}(x^p) = (p)_{q^{p^{m-1}}}x^{p-1}. \qedhere
\]
\end{enumerate}
\end{proof}

We can raise the level of a twisted derivation as follows:

\begin{prop} \label{pullb}
If $\theta'$ is a twisted derivation of level $-m$ on an $A'$-module $M'$, then there exists a unique derivation $\theta$ of level $-m+1$ on $M := A {}_{{}_{F}\nwarrow}\!\!\otimes_{A'} M'$ such that
\[
\forall s \in M', \quad \theta(1 \otimes s) = x^{p-1} \otimes \theta'(s).
\]
\end{prop}

\begin{proof}
Necessarily, we will have for $f \in A$ and $s \in M'$,
\begin{align*}
\theta(f \otimes s) &= (p^{m-1})_{q} \partial_{q^{p^{m-1}}}(f) \otimes s + x^{p-1}\sigma^{p^{m-1}}(f) \otimes \theta'(s)
\end{align*}
and we first need to show that this is well defined.
On the one hand, we have for $f \in A'$,
\begin{align*}
\theta(1 \otimes fs) &=x^{p-1} \otimes \theta'(fs)
\\ & = x^{p-1} \otimes (p^m)_{q} \partial_{q^{p^m}}(f)s + x^{p-1} \otimes \sigma^{p^m}(f)\theta'(s) \\
& = (p^m)_{q} x^{p-1} F(\partial_{q^{p^m}}(f)) \otimes s + x^{p-1} F(\sigma^{p^m}(f))\otimes \theta'(s)
\end{align*}
and on the other hand, we have
\begin{align*}
\theta(F(f) \otimes s) =(p^{m-1})_{q} \partial_{q^{p^{m-1}}}(F(f)) \otimes s + x^{p-1}\sigma^{p^{m-1}}(F(f)) \otimes \theta'(s).
\end{align*}
We may therefore rely on lemma \ref{commute} because $(p^{m-1})_{q} (p)_{q^{p^{m-1}}} = (p^m)_{q}$. 
It only remains to show that $\theta$ is a twisted derivation:
\begin{align*}
\theta(fg\otimes s) & = (p^{m-1})_{q} \partial_{q^{p^{m-1}}}(fg) \otimes s + x^{p-1}\sigma^{p^{m-1}}(fg) \otimes \theta'(s) \\
&= (p^{m-1})_{q} \partial_{q^{p^{m-1}}}(f)g \otimes s + (p^{m-1})_{q} \sigma^{p^{m-1}}(f)\partial_{q^{p^{m-1}}}(g) \otimes s \\
&+ x^{p-1}\sigma^{p^{m-1}}(f)\sigma^{p^{m-1}}(g)\otimes \theta'(s) \\
&= (p^{m-1})_{q} \partial_{q^{p^{m-1}}}(f)g \otimes s \\
&+\sigma^{p^{m-1}}(f)\left( (p^{m-1})_{q} \partial_{q^{p^{m-1}}}(g) \otimes s + x^{p-1}\sigma^{p^{m-1}}(g)\otimes \theta'(s)\right) \\
&=(p^{m-1})_{q} \partial_{q^{p^{m-1}}}(f)g \otimes s +\sigma^{p^{m-1}}(f)\theta(g \otimes s). \qedhere
\end{align*}
\end{proof}

\begin{rmks}
\begin{enumerate}
\item
Alternatively, a twisted connection $\nabla'$ of level $-m$ on $M'$ gives rise to a twisted connection of level $-m+1$ on $M := A {}_{{}_{F}\nwarrow}\!\!\otimes_{A'} M'$.
\item
This may also be interpreted in terms of $\mathcal D$-modules but the functor does \emph{not} come from a morphism between rings of differential operators.
It will be necessary to consider hyper-stratifications to go further.
\end{enumerate}
\end{rmks}

\begin{dfn} 
The functor
\begin{equation}\label{levlr}
F^* : \mathrm{MIC}_{q}^{(-m)}(A'/R) \to \mathrm{MIC}_{q}^{(-m+1)}(A/R), \quad (M', \theta') \mapsto (M, \theta)
\end{equation}
is the \emph{level raising} functor for twisted connections of negative level.
\end{dfn}

\begin{rmks}
\begin{enumerate}
\item
We showed in corollary 8.9 of \cite{GrosLeStumQuiros19} that the level raising functor 
\begin{equation}\label{levlr0}
F^* : \mathrm{MIC}_{q}^{(-1)}(A'/R) \to \mathrm{MIC}_{q}(A/R), \quad (M', \theta') \mapsto (M, \theta)
\end{equation}
induces an equivalence on topologically quasi-nilpotent objects objects killed by $(p)_{q}$.
This does not extend to higher $m$.
\item
Shiho proved in theorem 3.1 of \cite{Shiho15} that the level rising functor \eqref{levlr0} is also an equivalence when $q=1$ (the untwisted case) on topologically quasi-nilpotent objects.
Again, this does not extend to higher $m$.
\end{enumerate}
\end{rmks}

Up to the end of the section, we focus on the case $m=1$.

The divided frobenius introduced in definition \ref{divfrob} is $F$-linear with respect to the ``left'' structure.
One can show that it is also $F$-linear with respect to the ``right'' structure.
More precisely, if we denote by
\[
[F]_{n} : A'\langle \omega \rangle_{q(-1)}/I^{\{n+1\}_{q}} \to A\langle\xi \rangle_{q}/I^{[n+1]} \quad \left(\mathrm{resp.}\ \widehat{[F]} : \widehat{A'\langle\omega \rangle}_{q(-1)} \to \widehat{A\langle\xi \rangle}_{q} \right) 
\]
the \emph{divided frobenius of order $n$} (resp. the \emph{completed divided frobenius}), then we have:

\begin{lem}
The diagrams
\[
\xymatrix{
A' \ar[r]^-{\theta_{n}} \ar[d]^F & A'\langle \omega \rangle_{q(-1)}/I^{\{n+1\}_{q}} \ar[d]^{[F]_{n}} \\
A \ar[r]^-{\theta_{n}} &A\langle\xi \rangle_{q}/I^{[n+1]}
}
\quad \mathrm{and} \quad 
\xymatrix{
A' \ar[r]^-{\theta} \ar[d]^F & \widehat {A'\langle\omega\rangle}_{q(-1)} \ar[d]^{\widehat {[F]}} \\
A \ar[r]^-{\theta} & \widehat {A\langle \xi \rangle}_{q}
}
\]
are commutative.
\end{lem}

\begin{proof}
Using definition \ref{deftayl}, this follows from the commutativity of the various squares in the diagram (we only draw the completed case)
\[
\xymatrix{
A' \ar[r]^-{p_{2}} \ar[d]^F & P_{A'/R} \ar[r] \ar[d]^F & \varprojlim \widehat P_{A'/R}/\widehat I_{A'/R}^{(n+1)_{q^p}} \ar[r]^-\simeq \ar[d]^F & A'[[\xi]]_{q^p} \ar[r] \ar[d]^F & \widehat {A'\langle\omega\rangle}_{q(-1)} \ar[d]^{\widehat{[F]}}\\
A \ar[r]^-{p_{2}} & P_{A/R} \ar[r] & \varprojlim \widehat P_{A/R}/\widehat I_{A/R}^{(n+1)_{q}} \ar[r]^-\simeq & A[[\xi]]_{q} \ar[r] & \widehat {A\langle \xi \rangle}_{q},
}
\]
the last one coming from theorem \ref{propdivf}.
\end{proof}

The following definitions therefore make sense:

\begin{dfn}
\begin{enumerate}
\item
The \emph{level raising functor}
\[
F^* : \mathrm{Strat}_{q}^{(-1)}(A'/R) \to \mathrm{Strat}_{q}^{(0)}(A/R)
\]
is the functor obtained by pulling back the stratification along the divided frobenius of finite orders.
\item
The \emph{level raising functor}
\[
F^* : \widehat{\mathrm{Strat}}_{q}^{(-1)}(A'/R) \to \widehat{\mathrm{Strat}}_{q}^{(0)}(A/R)
\]
is the functor obtained by pulling back the hyper-stratification along the completed divided frobenius.
\end{enumerate}
\end{dfn}

\begin{rmks}
\begin{enumerate}
\item
Since $[F](\omega) \equiv x^{p-1}\xi \mod \xi^{[2]_{q}}$, one easily sees that the diagram
\[
\xymatrix{
\mathrm{Strat}_{q}^{(-1)}(A'/R) \ar[r]^{F^*} \ar[d]^{\simeq} & \mathrm{Strat}_{q}^{(0)}(A/R) \ar[d]^{\simeq}
\\ \mathrm{MIC}_{q}^{(-1)}(A'/R)\ar[r]^{F^*} & \mathrm{MIC}_{q}^{(0)}(A/R)
}
\]
is commutative.
\item
The diagram
\[
\xymatrix{
\widehat{\mathrm{Strat}}_{q}^{(-1)}(A'/R)\ar[r]^{F^*} \ar[d] & \widehat{\mathrm{Strat}}_{q}^{(0)}(A/R) \ar[d]
\\ \mathrm{Strat}_{q}^{(-1)}(A'/R)\ar[r]^{F^*} & \mathrm{Strat}_{q}^{(0)}(A/R)
}
\]
is obviously also commutative.
\end{enumerate}
\end{rmks}

\begin{prop}
The divided frobenius $[F] : A'\langle\omega \rangle_{q(-1)} \to A\langle\xi \rangle_{q} $ is free of degree $p^2$.
\end{prop}

\begin{proof}
Since $F : A' \to A$ is free of degree $p$, it is sufficient to prove that the $A$-linearization
\[
A {}_{{}_{F}\nwarrow}\!\!\otimes_{A'} A' \langle\omega \rangle_{q(-1)} \to A\langle\xi \rangle_{q}
\]
of $[F]$ is free on $1, \xi, \ldots, \xi^{p-1}$.
Thus, we have to show that $\{[F](\omega^{\{n\}_{q}})\xi^i\}_{0\leq i < p, n \in \mathbb Z_{\geq 0}}$ is a basis of $A\langle\xi \rangle_{q}$ as an $A$-module.
This follows from proposition 7.9 in \cite{GrosLeStumQuiros19} which gives the precise value of the leading coefficient
\[
b_{n,pn} = \prod_{k=1}^n\prod_{i=1}^{p-1} (kp-i)_{q} \in R^\times
\]
in the explicit formula recalled in the proof of theorem \ref{propdivf}.
\end{proof}

\begin{rmks}
\begin{enumerate}
\item
As a consequence, we see that the completed divided frobenius
\[
\widehat{[F]} : \widehat{A'\langle\omega \rangle}_{q(-1)} \to \widehat{A\langle\xi \rangle}_{q}
\]
is also finite free and in particular faithfully flat.
And the same is true for the morphism
\[
\widehat{[F]} \widehat\otimes \widehat{[F]} : \widehat{A'\langle\omega \rangle}_{q(-1)} \widehat\otimes_{A'} \widehat{A'\langle\omega \rangle}_{q(-1)} \to \widehat{A\langle\xi \rangle}_{q} \widehat\otimes_{A} \widehat{A\langle\xi \rangle}_{q}.
\]
\item The analogs at finite order are not true anymore.
\end{enumerate}
\end{rmks}

\begin{lem}
The composite map
\begin{equation} \label{compm}
A[\xi] \to A \otimes_{R} A \twoheadrightarrow A \otimes_{A'} A, \quad \xi \mapsto 1 \otimes x - x \otimes 1
\end{equation}
extends uniquely to a morphism of $\delta$-rings
\begin{equation} \label{desc}
u: A\langle\xi \rangle_{q} \to A \otimes_{A'} A
\end{equation}
and the diagram
\begin{equation} \label{magsq}
\xymatrix{
A'\langle \omega \rangle_{q(-1)} \ar[d]^{[F]} \ar[r] & A' \ar[d] \\
A\langle\xi \rangle_{q} \ar[r]^-u & A \otimes_{A'} A 
}
\end{equation}
is commutative (the upper map sends $\omega^{\{n\}_{q}}$ to $0$).
\end{lem}

\begin{proof}
Since
\[
\phi(\xi) \in A[\xi] \mapsto 1 \otimes x^p - x^p \otimes 1 \in A \otimes_{R} A \mapsto 0 \in A \otimes_{A'} A,
\]
the first assertion follows from the universal property of $A\langle\xi \rangle_{q}$ (theorem 3.6 in \cite{GrosLeStumQuiros20}).
For the second assertion, we have to show that the map $u$ sends $[F](\omega^{\{n\}_{q}})$ to $0$ for $n > 0$.
By functoriality, we may clearly assume that $R = \mathbb Z[q]_{p,q-1}$ and $A=R[x]$.
It is then sufficient to show that $(n)_{q^p}!(p)_{q}^n[F](\omega^{\{n\}_{q}})$ is sent to $0$ for $n > 0$ but this is exactly $\phi(\xi^{(n)})$ which is a multiple of $\phi(\xi)$ which itself, as we saw above, is sent to $0$ in $A \otimes_{A'} A$.
\end{proof}

\begin{rmk}
The existence of the map $u$ relies on the universal property of twisted divided powers and is not trivial at all:
when $p=2$, it sends $\xi^{[2]_{q}}$ to $x^2 (1 \otimes 1) - x \otimes x$ and for odd $p$, it sends $\xi^{[p]_{q}}$ to
\[
\frac {(1-q)(\frac {p-3}2)_{q^p}}{(p-1)_{q}!} x^p (1 \otimes 1) + \sum_{i=1}^{p-1} \frac {(-1)^{i} }{(i)_{q}!(p-i)_{q}!} q^{\frac {i(i-1)}2}x^{i} \otimes x^{p-i}.
\]
\end{rmk}

We can now prove Berthelot's frobenius descent (see theorem 2.3.6 in \cite{Berthelot00}) in our situation (from level $0$ to level $-1$):

\begin{thm} \label{mainthm}
Raising level provides an equivalence of categories
\[
\widehat{\mathrm{Strat}}_{q}^{(-1)}(A'/R)\simeq \widehat{\mathrm{Strat}}_{q}^{(0)}(A/R).
\]
\end{thm}

\begin{proof}
Recall that a descent datum along $F$ on an $A$-module $M$ is an isomorphism $\epsilon : A \otimes_{A'} M \simeq M \otimes_{A'} A$ satisfying the cocycle and normalization conditions.
If we denote by $(A/A')-\mathrm{Mod}$ the category of $A$-modules endowed with a descent data along $F$, then there exists an equivalence of categories
\[
A'-\mathrm{Mod} \simeq (A/A')-\mathrm{Mod}, \quad M' \mapsto (A \otimes_{A'} M', \mathrm{can}),
\]
where $\mathrm{can}$ denotes the canonical descent datum (\cite[\href{https://stacks.math.columbia.edu/tag/023F}{Section 023F}]{stacks-project}).
This happens because $F$ is faithfully flat (descent is effective).
On the other hand, pulling back an hyper-$q$-stratification along the map 
\[
u: A\langle\xi \rangle_{q} \to A \otimes_{A'} A,
\]
 provides a descent datum and we obtain a functor
\[
\widehat{\mathrm{Strat}}_{q}^{(0)}(A/R) \to (A/A')-\mathrm{Mod}.
\]
Diagram \eqref{magsq} provides us with a commutative diagram of functors
\[
\xymatrix{
\widehat{\mathrm{Strat}}_{q}^{(-1)}(A'/R) \ar[d] \ar[r] & A'-\mathrm{Mod} \ar[d]^-{\simeq} \\
\widehat{\mathrm{Strat}}_{q}^{(0)}(A/R) \ar[r] & (A/A')-\mathrm{Mod}.
}
\]
Now, let $M', N' \in \widehat{\mathrm{Strat}}_{q}^{(-1)}(A'/R)$ and
\[
u : M := A \otimes_{A'} M' \to N := A \otimes_{A'} N'
\]
be a morphism in $\widehat{\mathrm{Strat}}_{q}^{(0)}(A/R)$.
The map $u$ induces a morphism in $(A/A')-\mathrm{Mod}$.
Since descent along $F$ is effective, it comes from a unique $A'$-linear map $u': M' \to N'$.
Moreover, since $\widehat{[F]}$ is faithful, the diagram
\[
\xymatrix{
\widehat{A'\langle \omega \rangle}_{q(-1)} \otimes_{A'}' M' \ar[r]^\epsilon \ar[d]^{1 \otimes u'} & M' \otimes_{A'} \widehat{A'\langle \omega \rangle}_{q(-1)} \ar[d]^{u' \otimes 1}
\\ \widehat{A'\langle \omega \rangle}_{q(-1)} \otimes_{A'}' N' \ar[r]^\epsilon & N' \otimes_{A'} \widehat{A'\langle \omega \rangle}_{q(-1)}
}
\]
is commutative.
It follows that $u'$ is actually a morphism in $\widehat{\mathrm{Strat}}_{q}^{(-1)}(A'/R)$ and our functor is therefore fully faithful.
Now, if $M \in \widehat{\mathrm{Strat}}_{q}^{(0)}(A/R)$, it has a descent datum along $F$ and comes therefore from an $A'$-module $M'$.
Since $\widehat{[F]}$ is fully faithful, the hyper-$q$-stratification of $M$ comes from a a unique isomorphism
\[
\widehat{A'\langle \omega \rangle}_{q(-1)} \otimes_{A'}' M' \simeq M' \otimes_{A'} \widehat{A'\langle \omega \rangle}_{q(-1)}.
\]
The cocycle and normalization conditions then follow from the faithfulness of $\widehat{[F]} \widehat\otimes \widehat{[F]}$.
\end{proof}

We may now state the twisted Simpson correspondence in this setting (generalizing corollary 8.9 of \cite{GrosLeStumQuiros19}):

\begin{cor} \label{eqmod}
Raising level induces an equivalence between locally quasi-nilpotent finitely presented $A'$-modules endowed with a twisted connection of level $-1$ and locally quasi-nilpotent finitely presented $A$-modules endowed with a twisted connection of level $0$. \qed
\end{cor}

\section{Prisms and twisted divided powers of negative level}

The content of this section is quite similar to sections 3, 4 and 5 of \cite{GrosLeStumQuiros20}, but the proofs are usually easier.

We let $R$ be a $\delta$-ring with fixed rank one element $q$ and we assume that $R$ is actually a $\mathbb Z[q]_{(p,q-1)}$-algebra.
It does not really matter at this point, but we will also assume that $R$ is $(p)_{q}$-torsion free.
All $R$-algebras (or modules) are implicitly endowed with their $(p,q-1)$-adic topology and completion will always be meant with respect to this topology.
Finally, we will use here our notations from \cite{GrosLeStumQuiros20} and denote by $A^{\delta}$ the $\delta$-envelope of an $R$-algebra $A$ and by $I_{\delta}$ the $\delta$-envelope of an ideal $I$ (so that, for example, when $A = R[x]$ and $I=(x)$, we have $A^\delta = R[x_{0}, x_{1}, \ldots ]$ and $(x)_{\delta} = (x_{0}, x_{1}, \ldots)$).

We recall that, by definition, a \emph{$\delta$-pair} over $R$ is a couple $(B, J)$ made of a $\delta$-$R$-algebra $B$ and an ideal $J \subset B$.
With the obvious morphisms, $\delta$-pairs form a category.
When $J = (d)$ is a principal ideal, we will simply write $(B, d)$.
We introduce now the algebraic version of a prismatic envelope:

\begin{dfn}
The \emph{$(p)_{q}$-envelope} of a $\delta$-pair $(B, J)$ (if it exists) is a $\delta$-pair which is universal for morphisms from $(B,J)$ to $(p)_{q}$-torsion free $\delta$-pairs of the form $(B' ,(p)_{q})$.
\end{dfn}

This means that there exists a $(p)_{q}$-torsion free $\delta$-ring that we will usually denote by $B[J/(p)_{q}]^\delta$ and a morphism $B \to B[J/(p)_{q}]^\delta$ of $\delta$-rings such that $JB[J/(p)_{q}]^\delta \subset (p)_{q}B[J/(p)_{q}]^\delta$ satisfying the following universal property: any morphism of $\delta$-rings $B \to B'$ where $B'$ is $(p)_{q}$-torsion free and $JB' \subset (p)_{q}B'$ factors uniquely through $B[J/(p)_{q}]^\delta$.

\begin{rmks}
\begin{enumerate}
\item
Let $(B, J)$ be a $\delta$-pair and $\mathfrak b \subset B$ a $\delta$-ideal (meaning that it is stable under $\delta$).
Assume that $(B, J)$ has a $(p)_{q}$-envelope $B[J/(p)_{q}]^\delta$.
Then, if the quotient ring $B[J/(p)_{q}]^\delta/\mathfrak bB[J/(p)_{q}]^\delta$ is $(p)_{q}$-torsion free, this is the $(p)_{q}$-envelope of $(B/\mathfrak b, (J+\mathfrak b)/\mathfrak b)$.
\item
Let $\{(B_{e}, J_{e})\}_{e\in E}$ be a commutative diagram of $\delta$-pairs all having a $(p)_{q}$-envelope $B_{e}[J_{e}/(p)_{q}]^\delta$, $B := \varinjlim B_{e}$ and $J = \sum J_{e}B$.
If $\varinjlim B_{e}[J_{e}/(p)_{q}]^\delta$ is $(p)_{q}$-torsion free, then this is the $(q)_{p}$-envelope of $(B,J)$.
\end{enumerate}
\end{rmks}

\begin{xmps}
\begin{enumerate}
\item If $B$ is $(p)_{q}$-torsion free and $J \subset (p)_{q}B$, then $B[J/(p)_{q}]^\delta=B$.
\item If $J = B$, then $B[J/(p)_{q}]^\delta=B\left[\{1/(n)_{q}\}_{n \in \mathbb Z_{\geq 0}}\right]$.
\item Let $B := R[x]^\delta$ and $J := (x)_{\delta}$.
Then the $(p)_{q}$-envelope of $(B, J)$ is $R[w]^\delta$ with respect to the blowing up $R[x] \to R[w], x \mapsto (p)_{q}w$.
One can derive many other examples from this one.
\item Assume $B=R[x]$, $\delta(x) = 0$ and $p=2$.
Then,
\[
\delta\left(\frac x {(2)_{q}}\right) = \frac q{(2)_{q^2}} \left(\frac x {(2)_{q}}\right)^2 \notin R\left[\frac x {(2)_{q}}\right]
\]
(the $(p)_{q}$-envelope is not noetherian in general).
\end{enumerate}
\end{xmps}

\begin{thm} \label{pqenv}
Let $A$ be $p$-torsion free $\delta$-$R$-algebra with fixed rank one element $x$.
If we endow $A[\xi]$ with the \emph{symmetric} $\delta$-structure (with respect to $x$), then the $(p)_{q}$-envelope of $(A[\xi], \xi)$ is $A\langle \omega \rangle_{q(-1)}$.
\end{thm}

\begin{proof}
This is analog to theorem 3.6 of \cite{GrosLeStumQuiros20} and we will only give a sketch of the proof.
We may assume if we wish that $R = \mathbb Z[q]_{(p,q-1)}$ and $A = R[x]$.
The point is to show that, if we let $v_{n} = \prod(\delta^{r}(\omega))^{a_{r}}$ where $n = \sum a_{r}p^r$ denotes the $p$-adic expansion of a non-negative integer, then $\{v_{n}\}_{n \in \mathbb Z_{\geq 0}}$ is an alternative basis for $A\langle \omega \rangle_{q(-1)}$.
The universal property will then be automatic.
It is clearly sufficient to prove that
\begin{equation} \label{cong}
\delta^{r}(\omega) \equiv c_{r}\omega^{\{p^{r}\}_q} \mod F_{p^r-1} \quad \mathrm{with}\quad c_{r} \in R^\times
\end{equation}
where $F_{n}$ denotes the $A$-submodule generated by $\omega^{\{k\}_q}$ for $k \leq n$.

Let us first show that
\[
\delta\left(\omega^{\{p^r\}_q}\right) \equiv d_{r} \omega^{\{p^{r+1}\}_q} \quad \mathrm{and} \quad \phi\left(\omega^{\{p^r\}_q}\right) \equiv e_{r} \omega^{\{p^{r+1}\}_q}\ \mod F_{p^{r+1}-1}
\]
with $d_{r} \in R^\times$ and $e_{r} \notin R^\times$.
It is not difficult to see that there exists such congruences with $d_{r}, e_{r} \in R$.
Since $R$ is a local ring with maximal ideal $(p,q-1)$, it is then sufficient to prove the conditions $d_{r} \in R^\times$ and $e_{r} \notin R^\times$ modulo $(p,q-1)$.
Since $(q-1)$ is a $\delta$-ideal, we may therefore assume that $q= 1$ (so that $R = \mathbb Z_{p}$).
We can do the following computations
\[
\phi\left(\omega^{\{p^r\}}\right) \equiv \frac {p^{p^r(p-1)}p^{r+1}!}{p^r!} \omega^{\{p^{r+1}\}}, \quad \left(\omega^{\{p^r\}}\right)^p \equiv \frac {p^{r+1}!}{(p^r!)^p} \omega^{\{p^{r+1}\}} \mod F_{p^{r+1}-1}
\]
and check that the valuations of the coefficients are $p^{r+1} > 1$ (which shows that $v_{p}(e_{r}) > 0$) and $1$ respectively.
Finally, since
\[
\delta\left(\omega^{\{p^r\}}\right) =\frac {\phi\left(\omega^{\{p^r\}}\right) - \left(\omega^{\{p^r\}}\right)^p}p,
\]
we see that $v_{p}(d_{r}) = 0$.

We can now prove congruence \eqref{cong} by induction as follows:
\begin{align*}
\delta^{r+1}(\omega) & \equiv \delta\left(c_{r}\omega^{\{p^{r}\}_q}\right) \mod F_{p^{r+1}-1} \\
& \equiv c_{r}^p\delta\left(\omega^{\{p^r\}_q}\right) + \delta(c_{r})\phi\left(\omega^{\{p^r\}_q}\right) \mod F_{p^{r+1}-1}\\
& \equiv (c_{r}^pd_{r} + \delta(c_{r})e_{r})\omega^{\{p^{r+1}\}_q} \mod F_{p^{r+1}-1}\\
& \equiv c_{r+1} \omega^{\{p^{r+1}\}_q} \mod F_{p^{r+1}-1}. \qedhere
\end{align*}
\end{proof}

\begin{dfn} \label{bnd}
An $R$-algebra $B$ is said to be \emph{bounded} (with respect to $(p)_{q}$) if $B$ is $(p)_{q}$-torsion free and $B/(p)_{q}B$ has bounded $p^\infty$-torsion.
We call a $\delta$-pair $(B,J)$ \emph{bounded} when $B$ is bounded.
\end{dfn}

We assume from now on that $R$ itself is \emph{bounded}.
We may then recall the following fundamental definition:

\begin{dfn}[Bhatt-Scholze]
A \emph{bounded prism over $\left(\widehat R, (p)_{q}\right)$} (we will say a \emph{$(p)_{q}$-prism over $R$}) is a $\delta$-pair over $R$ of the form $(B, (p)_{q})$ where $B$ is complete and bounded.
\end{dfn}

\begin{rmks}
\begin{enumerate}
\item
There exists a more general notion of prism or bounded prism that we will not consider here.
\item
In \cite{BhattScholze19}, it is only required that $B$ is derived complete. Actually, when $B$ is bounded, derived completeness is equivalent to completeness, and we always assume here that the prisms are bounded.
\item 
A $(p)_{q}$-prism over $R$ is completely determined by the $\delta$-$R$-algebra $B$ and we may simply say that the ring $B$ is a $(p)_{q}$-prism over $R$.
However, it is fundamental to understand that $B$ is actually a \emph{$\delta$-thickening} with respect to $(p)_{q}$ of $\overline B := B/(p)_{q}$, which is somehow the important object.
\item If $B$ is a $(p)_{q}$-prism, then the ideal $(p)_{q}B$ is closed in $B$, or equivalently $\overline B$ is $p$-adically complete (proposition 4.3 of \cite{GrosLeStumQuiros20}).
\item If $B$ is a bounded $\delta$-$R$-algebra, then $\left(\widehat B, (p)_{q}\right)$ is a $(p)_{q}$-prism over $R$ (proposition 4.3 of \cite{GrosLeStumQuiros20} again).
\end{enumerate}
\end{rmks}

\begin{dfn}
The \emph{(bounded) prismatic envelope} of a $\delta$-pair $(B, J)$ over $R$ is a $(p)_{q}$-prism which is universal for morphisms from $(B,J)$ to $(p)_{q}$-prisms over $R$.
\end{dfn}

\begin{rmks}
\begin{enumerate}
\item There exists a more general notion of prismatic envelope for non bounded prisms but it is then necessary to use derived completions.
\item If $(B,J)$ has a $(p)_{q}$-envelope which is bounded, then the completion of this $(p)_{q}$-envelope is the prismatic envelope of $(B,J)$.
Note hat $B$ is automatically bounded when $B$ is flat (over $R$).
\item Prismatic envelopes are usually quite hard to compute (but see corollary \ref{prismenv} just below).
\end{enumerate}
\end{rmks}

Proposition \ref{pqenv} has the following consequence:

\begin{cor} \label{prismenv}
Let $A$ be bounded $\delta$-$R$-algebra with fixed rank one element $x$.
If we endow $A[\xi]$ with the \emph{symmetric} $\delta$-structure (with respect to $x$), then the prismatic envelope of $(A[\xi], (\xi))$ is $\widehat {A\langle \omega \rangle}_{q(-1)}$. \qed
\end{cor}

Finally, we also have the following:

\begin{prop} \label{univ}
Let $A$ be a complete bounded $R$-algebra with topologically \'etale coordinate $x$.
Then $\widehat {A\langle\omega\rangle}_{q(-1)}$ is the prismatic envelope of $(P_{A/R}, I_{A/R})$.
\end{prop}

\begin{proof}
Using corollary \ref{prismenv}, this is completely analog to the proof of theorem 5.2 in \cite{GrosLeStumQuiros20}.
\end{proof}

\section{Cartier descent and prismatic crystals} \label{despr}

We keep the same notations as before: $R$ is a $\delta$-ring, $q \in R$ satisfies $\delta(q)=0$ and we assume that $R$ is a $\mathbb Z[q]_{(p,q-1)}$-algebra.
We also assume now that $R$ is bounded (definition \ref{bnd}) so that $(\widehat R, (p)_{q})$ is a (bounded) prism. 

\begin{dfn}
If $\mathcal X$ is a $p$-adic formal scheme\footnote{We will only consider topologically smooth formal schemes - otherwise, the definition has to be modified.} over $R/(p)_{q}$, then the \emph{$(p)_{q}$-prismatic site of $\mathcal X/R$} is the category opposite\footnote{We'd rather stay close to the geometric feeling.} to the category of all $(p)_{q}$-prisms $B$ over $R$ endowed with a morphism $\mathrm{Spf}(B/(p)_{q}) \to \mathcal X$ over $R/(p)_{q}$.
\end{dfn}

We will denote this category by $(p)_{q}\mathrm{-PRIS}(\mathcal X/R)$ and endow it with the coarse topology.
Also, when $\mathcal X = \mathrm{Spf}(S)$ is affine, we will write $S/R$ instead of $\mathcal X/R$.

\begin{rmks}
\begin{enumerate}
\item Assume $S$ is a complete $p$-adic $R/(p)_{q}$-algebra.
Then an object of $(p)_{q}\mathrm{-PRIS}(S/R)$ is essentially a commutative diagram
\[
\xymatrix{R \ar@{->>}[d] \ar[rr] && B \ar@{->>}[dd] \\ R/(p)_{q} \ar[d] && \\ S \ar[rr] &&B/(p)_{q}}
\]
where $B$ is a complete bounded $\delta$-$R$-algebra.
Since $R$ is fixed, we can even simply use the one-line diagram $S \to B/(p)_{q} \twoheadleftarrow B$.
\item
One may define in general the \emph{prismatic site} $\mathrm{PRIS}$ which is the category opposite to the category of all bounded prisms (that we did not define here).
\item 
There is no final object in $\mathrm{PRIS}$ but we may consider the slice category $(p)_{q}\mathrm{-PRIS}_{/R}$ of all bounded prisms over $(\widehat R, (p)_{q})$.
This category is a lot simpler and actually (anti) equivalent to the category of all complete bounded $\delta$-$R$-algebras $B$.
\item
We may fiber the category $(p)_{q}\mathrm{-PRIS}_{/R}$ over the category of $p$-adic $R/(p)_{q}$-formal schemes, and then the category $(p)_{q}\mathrm{-PRIS}(\mathcal X/R)$ is exactly the fiber over $\mathcal X$ (but there is no final objet anymore).
\item
Our category $(p)_{q}\mathrm{-PRIS}(\mathcal X/R)$ is the same as (or more precisely the opposite to) the category $(\mathcal X/R)_{\mathbb {\Delta}}$ in definition 4.1 of \cite{BhattScholze19}.
We choose to stick to the classical notations.
\end{enumerate}
\end{rmks}

\begin{prop} \label{last}
Let $A$ be a complete $R$-algebra with rank one topologically \'etale coordinate $x$ and $\overline A := A/(p)_{q}$.
We denote by $S:= R[w]^\delta$ and we endow
\[
U := S \widehat \otimes_{R} A\ (\simeq \widehat {A[w]}^{\delta})
\]
with the unique $\delta$-structure such that $\delta(x) = (p)_{q}w$.
Then $U$ is a covering of (the final object of the topos associated to) $(p)_{q}\mathrm{-PRIS}(\overline A/R)$.
\end{prop}

\begin{proof}
This follows from various universal properties as in theorem 6.6 of \cite{BerthelotOgus78} for example (see also proposition 7.4 of \cite{GrosLeStumQuiros19}).
\end{proof}

\begin{rmk}
We will show in a forthcoming article that $A$ itself is a covering of the final object if we use the flat topology instead of the coarse topology.
\end{rmk}

A sheaf on $(p)_{q}\mathrm{-PRIS}(\mathcal X/R)$ is given by a family of sets $E_{B}$ for each $B \in (p)_{q}\mathrm{-PRIS}(\mathcal X/R)$ and a compatible family of transition maps $E_{B} \to E_{B'}$ for each morphism $B \to B'$.
The structural sheaf $\mathcal O_{\mathcal X/R}$ corresponds to the case $E_{B} = B$ and an $\mathcal O_{\mathcal X/R}$-module is given by a family of $B$-modules $E_{B}$ together with (linear) transition maps $C \otimes_{B} E_{B} \to E_{C}$.

\begin{dfn}
A \emph{$(p)_{q}$-prismatic crystal} on $\mathcal X/R$ is a sheaf of $\mathcal O_{\mathcal X/R}$-modules on $(p)_{q}\mathrm{-PRIS}(\mathcal X/R)$ with bijective transition maps $C \otimes_{B} E_{B} \simeq E_{C}$.
\end{dfn}

\begin{rmk}
Alternatively, this means that
\[
\mathcal O_{\mathcal X/R} \otimes_{\Gamma(\mathcal X/R, O_{\mathcal X/R})} \Gamma(\mathcal X/R, E) \simeq E.
\]
\end{rmk}

Definition \ref{hyperq} can be generalized (it follows from proposition \ref{univ} that this is indeed a generalization):

\begin{dfn}
Let $B$ be a complete bounded $\delta$-$R$-algebra and $C$ the prismatic envelope of the diagonal in $B \otimes_{R} B$.
A \emph{twisted hyper-stratification of level $-1$} on a $B$-module $M$ is a $C$-linear isomorphism
\[
\epsilon : C \otimes'_{B} M \simeq M \otimes_{B} C
\]
satisfying the usual normalization and cocycle conditions.
\end{dfn}

Note that, as we have already remarked, it is usually very hard to compute $C$.
We will denote by $\widehat{\mathrm{Strat}}_{q}^{(-1)}(B/R)$ the category of $B$-modules endowed with a twisted hyper-stratification of level $-1$.

\begin{prop} \label{pris1}
Let $\mathcal X$ be a formal $R/(p)_{q}$-scheme, $B \in \mathrm{PRIS}(\mathcal X/R)$ and $C$ the prismatic envelope of the diagonal in $B \otimes_{R} B$.
If $E$ is a $(p)_{q}$-prismatic crystal on $\mathcal X/R$, then there exists a unique twisted hyper-stratification of level $-1$ on $E_{B}$ given by
\[
\epsilon : C \otimes'_{B} E_{B} \simeq E_{C} \simeq E_{B} \otimes_{B} C
\]
\end{prop}

\begin{proof}
The cocycle and normalization conditions are automatic by functoriality.
\end{proof}

This defines a functor
\[
\{(p)_{q}\mathrm{-prismatic}\ \mathrm{crystals}\ \mathrm{on}\ \mathcal X/R\} \to \widehat{\mathrm{Strat}}_{q}^{(-1)}(B/R).
\]

As a consequence, we see that the notion of a prismatic crystal is closely related to the constructions made in section \ref{neg}:

\begin{cor} \label{pris1}
Let $A$ be a complete $R$-algebra with rank one topologically \'etale coordinate $x$ and $\overline A = A/(p)_{q}$.
Then, there exists a functor
\[
\{(p)_{q}\mathrm{-prismatic}\ \mathrm{crystals}\ \mathrm{on}\ \overline A /R\} \to \widehat{\mathrm{Strat}}_{q}^{(-1)}(A/R) \to \mathrm{MIC}_{q}^{(-1)}(A/R)
\]
sending $E$ to $E_{A}$. \qed
\end{cor}

\begin{rmk}
We will show in a forthcoming article that the first functor is an equivalence, and consequently, that the composite map is also an equivalence if we restrict to locally finitely presented crystals and finitely presented modules with a topologically quasi-nilpotent connection.
\end{rmk}

\begin{cor}
Let $A$ be a complete $R$-algebra with rank one topologically \'etale coordinate $x$.
We denote by $S:= R[w]^\delta$ and we endow $U := S \widehat \otimes_{R} A\ (\simeq \widehat {A[w]}^{\delta})$ with the unique $\delta$-structure such that $\delta(x) = (p)_{q}w$.
Then there exists an equivalence of categories
\[
\{(p)_{q}\mathrm{-prismatic}\ \mathrm{crystals}\ \mathrm{on}\ \overline A/R\} \simeq \widehat{\mathrm{Strat}}_{q}^{(-1)}(U/R)
\]
\end{cor}

\begin{proof}
Follows from proposition \ref{last}.
\end{proof}

\begin{rmk}
Although this will follow from further investigation, we cannot prove directly that the functor
\[
\widehat{\mathrm{Strat}}_{q}^{(-1)}(U/R) \to \widehat{\mathrm{Strat}}_{q}^{(-1)}(A/R)
\]
induced by the morphism of $\delta$-$R$-algebras $\pi : U \to A, w \mapsto 0$ is an equivalence.
There exists an obvious section $s$ of the morphism $\pi$ but it is not a $\delta$-morphism.
Even more: if we denote by
\[
V := \widehat{(U \otimes_{R} U)[\omega]^\delta/((p)_{q}\omega - \xi)_{\delta}}
\]
the prismatic envelope of the diagonal in $U \otimes_{R} U$, then the section $s$ does not extend naturally to a section of the morphism $V \to A\langle \omega \rangle_{q(-1)}, w \mapsto 0$.
\end{rmk}

Recall now that we may also consider the $q$-crystalline site $q\mathrm{-CRIS}$ as we did in section 7 of \cite{GrosLeStumQuiros20} as well as its slice category $q\mathrm{-CRIS}_{/R}$ when $R$ is endowed with a $q$-PD-ideal $\mathfrak r$.
More generally, when $\mathcal X$ is a formal $R/\mathfrak r$-scheme, there exists the $q$-crystalline site $q\mathrm{-CRIS}(\mathcal X/R)$.
This is the category opposite to the category of complete bounded $q$-PD-pairs (see definition 3.1 of \cite{GrosLeStumQuiros20}) $(B, J)$ over $R$ together with a morphism of formal $R/\mathfrak r$-schemes $\mathrm{Spf}(B/J) \to \mathcal X$.
There exists a corresponding notion of a $q$-crystal and, with $A$ as in lemma \ref{pris1} and $\mathfrak a$ a closed $q$-PD-ideal containing $\mathfrak rA$, if $E$ is a $q$-crystal on $(A/\mathfrak a)/R$, then we may endow $E_{A}$ with a twisted hyper-stratification of level $0$.

Let us mention (as this will be needed in the next definition) that if $(B, J)$ is a $q$-PD-pair, then $\phi(J) \subset (p)_{q}B$ and the frobenius therefore induces a morphism $\overline \phi : B/J \to B/(p)_{q}$.
We may now recall from the proof of theorem 16.17 of \cite{BhattScholze19} the following construction (we use the notation $\widetilde {\ }$ to denote the associated topos):

\begin{dfn} \label{carmor}
Let $(R, \mathfrak r)$ be a $q$-PD-pair, $\mathcal X$ a formal $R/\mathfrak r$-scheme and $\mathcal X' := \mathcal X \widehat \otimes_{{}_{R/\mathfrak r} \nearrow\overline \phi} R/(p)_{q}$.
Then the morphism of topos
\[
C_{\mathcal X/R} : \widetilde{q\mathrm{-CRIS}}(\mathcal X/R) \to \widetilde{(p)_{q}-\mathrm{PRIS}}(\mathcal X'/R).
\]
induced by the functor that sends
\begin{enumerate}
\item a $q$-PD-pair $(B, J)$ to the prism $(B, (p)_{q})$,
\item a morphism $\mathrm{Spf}(B/J) \to \mathcal X$ to the linearization $\mathrm{Spf}(B/(p)_{q}) \to \mathcal X'$ of the composite map 
\[
\mathrm{Spf}(B/(p)_{q}) \to \mathrm{Spf}(B/J) \to \mathcal X,
\]
\end{enumerate}
is the \emph{prismatic Cartier morphism} of $\mathcal X/R$.
\end{dfn}

When $\mathcal X = \mathrm{Spf}(S)$, we may also write $C_{S/R}$ for this morphism.

\begin{rmks}
\begin{enumerate}
\item
If $E$ is a $(p)_{q}$-prismatic sheaf on $\mathcal X'/R$ and $(B, J)$ is a (complete bounded) $q$-PD-pair, then we simply have $(C_{\mathcal X/R}^{-1}E)_{B} = E_{B}$.
\item
As a particular case, we see that $C_{\mathcal X/R}^{-1}\mathcal O_{\mathcal X'/R}= \mathcal O_{\mathcal X/R}$ so that $C_{\mathcal X/S}$ is trivially a (flat) morphism of ringed topos.
It is also clear that $C_{\mathcal X/R}^{-1}$ preserves crystals.
\end{enumerate}
\end{rmks}

\begin{xmp}
Let $(R, \mathfrak r) \to (A,\mathfrak a)$ be a morphism of $q$-PD-pairs with $A$ complete and bounded and $\mathfrak a$ closed in $A$.
If $\mathcal X := \mathrm{Spf}(A/\mathfrak a)$ and we let $A' := R {}_{{}_{\phi}\nwarrow}\!\!\widehat\otimes_{R} A$, then $\mathcal X' = \mathrm{Spf}(A'/(p)_{q})$
and we have
\[
C_{\mathcal X/R}\left(\vcenter{\xymatrix{R \ar[d] \ar[r] & A \ar@{->>}[d] \\ A/\mathfrak a \ar@{=}[r] & A/\mathfrak a}}\right) = \left(\vcenter{\xymatrix{R \ar[d] \ar[r] & A \ar@{->>}[d] \\ A'/(p)_{q} \ar[r]^{\overline F} &A/(p)_{q}}}\right)
\]
in which $F$ denotes as usual the relative frobenius of $A$.
\end{xmp}

\begin{prop} \label{comdia}
Let $A$ be a complete $R$-algebra with rank one topologically \'etale coordinate $x$ and $\mathfrak a$ a closed $q$-PD-ideal in $A$.
If $\overline A := A/\mathfrak a$ and $\overline A' := A'/(p)_{q}$ with $A' := R {}_{{}_{\phi}\nwarrow}\!\!\widehat\otimes_{R} A$, then the diagram
\begin{equation} \label{diagram}
\xymatrix{
\{(p)_{q}\mathrm{-prismatic}\ \mathrm{crystals}\ \mathrm{on}\ \overline A'/R\} \ar[r]^-{C_{\overline A/R}^{-1}} \ar[d] & \{q\mathrm{-crystals}\ \mathrm{on}\ \overline A/R\} \ar[d] \\
\widehat{\mathrm{Strat}}_{q}^{(-1)}(A'/R) \ar[r]^{F^*} & \widehat{\mathrm{Strat}}_{q}^{(0)}(A/R)
}
\end{equation}
is commutative.
\end{prop}

\begin{proof}
There exists a (vertical) morphism of $(p)_{q}$-prisms on $\overline A'/R$:
\[
\xymatrix{
A'/(p)_{q} \ar@{=}[r] \ar@{=}[d] & A'/(p)_{q} \ar[d]^{\overline F}& \ar@{->>}[l] A' \ar[d]^F\\ A'/(p)_{q} \ar[r]^{\overline F} & A/(p)_{q} & \ar@{->>}[l] A
}
\]
It follows that if $E$ is a $(p)_{q}$-prismatic crystal on $\overline A'/R$, then we have
\[
(C_{\overline A/R}^{-1}E)_{A} = E_{A} \simeq A {}_{{}_{F}\nwarrow}\!\!\otimes_{A'} E_{A'}.
\]
In the same way, the divided frobenius $[F] : A\langle \omega \rangle_{q(-1)} \to A\langle \xi \rangle_{p}$ induces a morphism of $(p)_{q}$-prisms on $\overline A'/R$.
It follows that the stratification on $E_{A'}$ comes from the $q$-crystalline structure of $C_{\overline A/R}^{-1}E$.
\end{proof}

\begin{rmks}
\begin{enumerate}
\item We proved in theorem \ref{mainthm} that the level rising map $F^*$ in the bottom of diagram \eqref{diagram} is an equivalence.
\item We will show in a forthcoming article that both vertical maps are an equivalence.
As a consequence, we will obtain for free that the upper map is also an equivalence.
\item  One may hope to prove more generally that the Cartier morphism of definition \ref{carmor} is always an equivalence of topos much as in \cite{Oyama17} or \cite{Xu19}.
\end{enumerate}
\end{rmks}

\addcontentsline{toc}{section}{References}
\printbibliography

\end{document}